\documentclass{article}
\usepackage{amsmath,amsthm,amssymb,graphicx}

\usepackage{tabu}
\newcommand{\scaleeq}[2]{\scalebox{#1}{$\displaystyle #2$}}
\newcommand{\eqsmall}[1]{\scaleeq{0.8}{#1}}

\usepackage[colorlinks=true,linkcolor=blue,citecolor=blue,pdfpagelabels=false]{hyperref}

\usepackage{thmtools}
\theoremstyle{plain}
  \declaretheorem[numberwithin=section]{theorem}
  \declaretheorem[numberlike=theorem]{corollary}
  
  \declaretheorem[numberlike=theorem]{lemma}
  \declaretheorem[numberlike=theorem]{conjecture}
\theoremstyle{definition}
  \declaretheorem[numberlike=theorem,qed=$\diamond$]{example}
  \declaretheorem[numberlike=theorem]{remark}

\newcommand{\assign}{:=}
\newcommand{\email}[1]{{\textit{Email:} \texttt{#1}}}
\newcommand{\rightarrowlim}{\mathop{\rightarrow}\limits}
\newcommand{\tmdummy}{$\mbox{}$}

\renewcommand{\geq}{\geqslant}
\renewcommand{\leq}{\leqslant}

\begin{document}

\title{Divisibility properties of sporadic Ap\'ery-like numbers}

\author{Amita Malik\thanks{\email{amalik10@illinois.edu}}\; and Armin Straub\thanks{\email{astraub@illinois.edu}}\\[0.4em]
Department of Mathematics\\
University of Illinois at Urbana-Champaign}

\date{August 1, 2015}
\maketitle

\begin{abstract}
In 1982, Gessel showed that the Ap\'ery numbers associated to the
irrationality of $\zeta(3)$ satisfy Lucas congruences. Our main result is to
prove corresponding congruences for all sporadic Ap\'ery-like sequences. In
several cases, we are able to employ approaches due to McIntosh, Samol--van
Straten and Rowland--Yassawi to establish these congruences. However, for the
sequences often labeled $s_{18}$ and $(\eta)$ we require a finer analysis.

As an application, we investigate modulo which numbers these sequences are
periodic. In particular, we show that the Almkvist--Zudilin numbers are
periodic modulo $8$, a special property which they share with the Ap\'ery
numbers. We also investigate primes which do not divide any term of a given
Ap\'ery-like sequence.
\end{abstract}

\section{Introduction}

In his surprising proof \cite{apery}, \cite{alf} of the irrationality of
$\zeta ( 3)$, R.~Ap\'ery introduced the sequence
\begin{equation}
  A (n) = \sum_{k = 0}^n \binom{n}{k}^2 \binom{n + k}{k}^2, \label{eq:A}
\end{equation}
which has since been referred to as the Ap\'ery sequence. It was shown by
I.~Gessel \cite[Theorem 1]{gessel-super} that, for any prime $p$, these
numbers satisfy the {\emph{Lucas congruences}}
\begin{equation}
  A ( n) \equiv A ( n_0) A ( n_1) \cdots A ( n_r) \pmod{p},
  \label{eq:lucas:i}
\end{equation}
where $n = n_0 + n_1 p + \cdots + n_r p^r$ is the expansion of $n$ in base
$p$. Initial work of F.~Beukers \cite{beukers-dwork} and D.~Zagier
\cite{zagier4}, which was extended by G.~Almkvist, W.~Zudilin
\cite{az-de06} and S.~Cooper \cite{cooper-sporadic}, has complemented the
Ap\'ery numbers with a, conjecturally finite, set of sequences, known as
Ap\'ery-like, which share (or are believed to share) many of the remarkable
properties of the Ap\'ery numbers, such as connections to modular forms
\cite{sb-picardfuchs}, \cite{beukers-apery87}, \cite{ahlgren-ono-apery}
or supercongruences \cite{beukers-apery85}, \cite{coster-sc},
\cite{ccs-apery}, \cite{os-supercong11}, \cite{os-domb},
\cite{oss-sporadic}. After briefly reviewing Ap\'ery-like sequences in
Section~\ref{sec:aperylike}, we prove in Sections~\ref{sec:lucas} and
\ref{sec:lucas2} our main result that all of these sequences also satisfy the
Lucas congruences \eqref{eq:lucas:i}. For all but two of the sequences, we
establish these congruences in Section~\ref{sec:lucas} by extending a general
approach provided by R.~McIntosh \cite{mcintosh-lucas}. The main difficulty,
however, lies in establishing these congruences for the sequence $(\eta)$. For
this sequence, and to a lesser extent for the sequence $s_{18}$, we require a
much finer analysis, which is given separately in Section~\ref{sec:lucas2}.

In the approaches of Gessel and McIntosh, binomial sums, like \eqref{eq:A},
are used to derive Lucas congruences. Other known approaches to proving Lucas
congruences for a sequence $C (n)$ are based on expressing $C (n)$ as the
constant terms of powers of a Laurent polynomial or as the diagonal
coefficients of a multivariate algebraic function. However, neither of these
approaches is known to apply, for instance, to the sequence $(\eta)$. In the
first approach, one seeks a Laurent polynomial $\Lambda (\boldsymbol{x}) =
\Lambda (x_1, \ldots, x_d)$ such that $C ( n)$ is the constant term of
$\Lambda ( \boldsymbol{x})$. In that case, we write $C (n) = \operatorname{ct} \Lambda
(\boldsymbol{x})^n$ for brevity. If the Newton polyhedron of $\Lambda (
\boldsymbol{x})$ has the origin as its only interior integral point, the results
of K.~Samol and D.~van Straten \cite{sd-laurent09} (see also
\cite{mv-laurent13}) apply to show that $C (n)$ satisfies the {\emph{Dwork
congruences}}
\begin{equation}
  C (p^r m + n) C (\lfloor n / p \rfloor) \equiv C (p^{r - 1} m + \lfloor n /
  p \rfloor) C (n) \pmod{p^r} \label{eq:dwork}
\end{equation}
for all primes $p$ and all integers $m, n \geq 0$, $r \geq 1$. The
case $r = 1$ of these congruences is equivalent to the Lucas congruences
\eqref{eq:lucas:i} for the sequence $C (n)$. For instance, in the case of the
Ap\'ery numbers \eqref{eq:A}, we have \cite[Remark~1.4]{s-apery}
\begin{equation*}
  A (n) = \operatorname{ct} \left[ \frac{(x + y) (z + 1) (x + y + z) (y + z + 1)}{x y
   z} \right]^n,
\end{equation*}
from which one may conclude that the Ap\'ery numbers satisfy the congruences
\eqref{eq:dwork}, generalizing \eqref{eq:lucas:i}. Similarly, for the sequence
$(\eta)$, one may derive from the binomial sum \eqref{eq:eta:s}, using
G.~Egorychev's method of coefficients \cite{egorychev-sums}, that its $n$th
term is given by $\operatorname{ct} \Lambda ( x, y, z)^n$, where
\begin{equation*}
  \Lambda ( x, y, z) = \left( 1 - \frac{1}{x y (1 + z)^5} \right)  \frac{(1 +
   x) (1 + y) (1 + z)^4}{z^3} .
\end{equation*}
However, $\Lambda ( x, y, z)$ is not a Laurent polynomial, and it is unclear
if and how one could express the sequence $( \eta)$ as constant terms of
powers of an appropriate Laurent polynomial. As a second general approach,
E.~Rowland and R.~Yassawi \cite{ry-diag13} show that Lucas congruences hold
for a certain class of sequences that can be represented as the diagonal
Taylor coefficients of $1 / Q (\boldsymbol{x})^{1 / s}$, where $s \geq 1$
is an integer and $Q (\boldsymbol{x}) \in \mathbb{Z} [\boldsymbol{x}]$ is a
multivariate polynomial. Again, while such representations are known for some
Ap\'ery-like sequences, see, for instance, \cite{s-apery}, no suitable
representations are available for the sequences $( \eta)$ or $s_{18}$.

It was conjectured by S.~Chowla, J.~Cowles and M.~Cowles \cite{ccc-apery}
and subsequently proven by I.~Gessel \cite{gessel-super} that
\begin{equation}
  A ( n) \equiv \left\{ \begin{array}{ll}
    1, & \text{if $n$ is even,}\\
    5, & \text{if $n$ is odd,}
  \end{array} \right. \pmod{8} . \label{eq:A:8}
\end{equation}

The congruences \eqref{eq:A:8} show that the Ap\'ery numbers are periodic
modulo $8$, and it was recently demonstrated by E.~Rowland and R.~Yassawi
\cite{ry-diag13} that they are not eventually periodic modulo $16$, thus
answering a question of Gessel. The Ap\'ery numbers are also periodic modulo
$3$ (see \eqref{eq:A:3}) and their values modulo $9$ are characterized by an
extension of the Lucas congruences \cite{gessel-super}; see also the recent
generalizations \cite{km-mod9} of C.~Krattenthaler and T.~M{\"u}ller, who
characterize generalized Ap\'ery numbers modulo $9$. As an application of
the Lucas congruences established in Sections~\ref{sec:lucas} and
\ref{sec:lucas2}, we address in Section~\ref{sec:periodic} the natural
question to which extent results like \eqref{eq:A:8} are true for
Ap\'ery-like numbers in general. In particular, we show in
Theorem~\ref{thm:AZ8} that the Almkvist--Zudilin numbers are periodic modulo
$8$ as well.

The primes $2, 3, 7, 13, 23, 29, 43, 47, \ldots$ do not divide any Ap\'ery
number $A ( n)$, and E.~Rowland and R.~Yassawi \cite{ry-diag13} pose the
question whether there are infinitely many such primes. While this question
remains open, we offer numerical and heuristic evidence that a positive
proportion of the primes, namely, about $e^{- 1 / 2} \sim 0.6065$, do not
divide any Ap\'ery number. In Section~\ref{sec:primes}, we investigate the
analogous question for other Ap\'ery-like numbers, and prove that Cooper's
sporadic sequences \cite{cooper-sporadic} behave markedly differently.
Indeed, for any given prime $p$, a fixed proportion of the last of the first
$p$ terms of these sequences is divisible by $p$. In the case of sums of
powers of binomial coefficients, such a result has been proven by N.~Calkin
\cite{calkin-bin}.

\section{Review of Ap\'ery-like numbers}\label{sec:aperylike}

Along with the Ap\'ery numbers $A (n)$, defined in \eqref{eq:A},
R.~Ap\'ery also introduced the sequence
\begin{equation*}
  B ( n) = \sum_{k = 0}^n \binom{n}{k}^2 \binom{n + k}{k},
\end{equation*}
which allowed him to (re)prove the irrationality of $\zeta ( 2)$. This
sequence is the solution of the three-term recursion
\begin{equation}
  (n + 1)^2 u_{n + 1} = (a n^2 + a n + b) u_n - c n^2 u_{n - 1},
  \label{eq:rec2-abc}
\end{equation}
with the choice of parameters $(a, b, c) = (11, 3, - 1)$ and initial
conditions $u_{- 1} = 0$, $u_0 = 1$. Because we divide by $(n + 1)^2$ at each
step, it is not to be expected that the recursion \eqref{eq:rec2-abc} should
have an integer solution. Inspired by F.~Beukers \cite{beukers-dwork},
D.~Zagier \cite{zagier4} conducted a systematic search for other choices of
the parameters $( a, b, c)$ for which the solution to \eqref{eq:rec2-abc},
with initial conditions $u_{- 1} = 0$, $u_0 = 1$, is integral. After
normalizing, and apart from degenerate cases, he discovered four
hypergeometric, four Legendrian as well as six sporadic solutions. It is still
open whether further solutions exist or even that there should be only
finitely many solutions. The six sporadic solutions are reproduced in
Table~\ref{tbl:sporadic2}. Note that each binomial sum included in this table
certifies that the corresponding sequence indeed consists of integers.

\begin{table}[h]\centering
  {\tabulinesep=1mm
  \begin{tabu}{|l|c|c|l|}
    \hline
    $(a, b, c)$ & \cite{zagier4} & \cite{asz-clausen} & $A (n)$ \\
    \hline
    $(7, 2, - 8)$ & {\textbf{A}} & (a) & \eqsmall{\sum_k \binom{n}{k}^3} \\
    \hline
    $(11, 3, - 1)$ & {\textbf{D}} & (b) & \eqsmall{\sum_k \binom{n}{k}^2 \binom{n + k}{n}} \\
    \hline
    $(10, 3, 9)$ & {\textbf{C}} & (c) & \eqsmall{\sum_k \binom{n}{k}^2 \binom{2 k}{k}} \\
    \hline
    $(12, 4, 32)$ & {\textbf{E}} & (d) & \eqsmall{\sum_k \binom{n}{k} \binom{2 k}{k} \binom{2 (n - k)}{n - k}} \\
    \hline
    $(9, 3, 27)$ & {\textbf{B}} & (f) & \eqsmall{\sum_k (- 1)^k 3^{n - 3 k} \binom{n}{3 k} \frac{(3 k) !}{k!^3}} \\
    \hline
    $(17, 6, 72)$ & {\textbf{F}} & (g) & \eqsmall{\sum_{k, l} (- 1)^k 8^{n - k} \binom{n}{k} \binom{k}{l}^3} \\
    \hline
  \end{tabu}}
  \caption{\label{tbl:sporadic2}The six sporadic solutions of \eqref{eq:rec2-abc}}
\end{table}

Similarly, the Ap\'ery numbers $A (n)$, defined in \eqref{eq:A}, are the
solution of the three-term recurrence
\begin{equation}
  (n + 1)^3 u_{n + 1} = (2 n + 1) (a n^2 + a n + b) u_n - n (c n^2 + d) u_{n -
  1}, \label{eq:rec3-abcd}
\end{equation}
with the choice of parameters $(a, b, c, d) = (17, 5, 1, 0)$ and initial
conditions $u_{- 1} = 0$, $u_0 = 1$. Systematic computer searches for further
integer solutions have been performed by G.~Almkvist and W.~Zudilin
\cite{az-de06} in the case $d = 0$ and, more recently, by S.~Cooper
\cite{cooper-sporadic}, who introduced the additional parameter $d$. As in
the case of \eqref{eq:rec2-abc}, apart from degenerate cases, only finitely
many sequences have been discovered. In the case $d = 0$, there are again six
sporadic sequences, which are recorded in Table~\ref{tbl:sporadic3}. Moreover,
by general principles (see \cite[Eq.~(17)]{cooper-sporadic}), each of the
sequences in Table~\ref{tbl:sporadic2} times $\binom{2 n}{n}$ is an integer
solution of \eqref{eq:rec3-abcd} with $d \neq 0$. Apart from such expected
solutions, Cooper also found three additional sporadic solutions, including
\begin{equation}
  s_{18} (n) = \sum^{\lfloor n / 3 \rfloor}_{k = 0} (- 1)^k \binom{n}{k}
  \binom{2 k}{k} \binom{2 (n - k)}{n - k} \left[ \binom{2 n - 3 k - 1}{n} +
  \binom{2 n - 3 k}{n} \right], \label{eq:s18}
\end{equation}
for $n \geq 1$, with $s_{18} (0) = 1$, as well as $s_7$ and $s_{10}$,
which are included in Table~\ref{tbl:sporadic3}. Remarkably, these sequences
are again connected to modular forms \cite{cooper-sporadic} (the subscript
refers to the level) and satisfy supercongruences, which are proved in
\cite{oss-sporadic}. Indeed, it was the corresponding modular forms and
Ramanujan-type series for $1 / \pi$ that led Cooper to study these sequences,
and the binomial expressions for $s_7$ and $s_{18}$ were found subsequently by
Zudilin (sequence $s_{10}$ was well-known before).

\begin{table}[h]\centering
  {\tabulinesep=1mm
  \begin{tabu}{|l|c|l|}
    \hline
    $(a, b, c, d)$ & \cite{asz-clausen} & $A (n)$\\
    \hline
    $(7, 3, 81, 0)$ & ($\delta$) & \eqsmall{\sum_k (- 1)^k 3^{n - 3 k} \binom{n}{3 k} \binom{n + k}{n} \frac{(3 k) !}{k!^3}} \\
    \hline
    $(11, 5, 125, 0)$ & ($\eta$) & \eqsmall{\sum_{k = 0}^{\lfloor n / 5 \rfloor} (- 1)^k \binom{n}{k}^3 \left( \binom{4 n - 5 k - 1}{3 n} + \binom{4 n - 5 k}{3 n} \right)} \\
    \hline
    $(10, 4, 64, 0)$ & ($\alpha$) & \eqsmall{\sum_k \binom{n}{k}^2 \binom{2 k}{k} \binom{2 (n - k)}{n - k}} \\
    \hline
    $(12, 4, 16, 0)$ & ($\epsilon$) & \eqsmall{\sum_k \binom{n}{k}^2 \binom{2 k}{n}^2} \\
    \hline
    $(9, 3, - 27, 0)$ & ($\zeta$) & \eqsmall{\sum_{k, l} \binom{n}{k}^2 \binom{n}{l} \binom{k}{l} \binom{k + l}{n}} \\
    \hline
    $(17, 5, 1, 0)$ & ($\gamma$) & \eqsmall{\sum_k \binom{n}{k}^2 \binom{n + k}{n}^2} \\
    \hline
    $(13, 4, - 27, 3)$ & $s_7$ & \eqsmall{\sum_k \binom{n}{k}^2 \binom{n + k}{k} \binom{2 k}{n}} \\
    \hline
    $(6, 2, - 64, 4)$ & $s_{10}$ & \eqsmall{\sum_k \binom{n}{k}^4} \\
    \hline
    $(14, 6, 192, - 12)$ & $s_{18}$ & defined in \eqref{eq:s18}\\
    \hline
  \end{tabu}}
  \caption{\label{tbl:sporadic3}The sporadic solutions of \eqref{eq:rec3-abcd}}
\end{table}

\section{Lucas congruences}\label{sec:lucas}

It is a well-known and beautiful classical result of Lucas \cite{lucas78}
that the binomial coefficients satisfy the congruences
\begin{equation}
  \binom{n}{k} \equiv \binom{n_0}{k_0} \binom{n_1}{k_1} \cdots
  \binom{n_r}{k_r} \pmod{p}, \label{eq:lucas:bin}
\end{equation}
where $p$ is a prime and $n_i$, respectively $k_i$, are the $p$-adic digits of
$n$ and $k$. That is, $n = n_0 + n_1 p + \cdots + n_r p^r$ and $k = k_0 + k_1
p + \cdots + k_r p^r$ are the expansions of $n$ and $k$ in base $p$.
Correspondingly, a sequence $a (n)$ is said to satisfy {\emph{Lucas
congruences}}, if the congruences
\begin{equation}
  a ( n) \equiv a ( n_0) a ( n_1) \cdots a ( n_r) \pmod{p}
  \label{eq:lucas}
\end{equation}
hold for all primes $p$. It was shown by I.~Gessel \cite[Theorem
1]{gessel-super} that the Ap\'ery numbers $A ( n)$, defined in
\eqref{eq:A}, satisfy Lucas congruences. E.~Deutsch and B.~Sagan
\cite[Theorem~5.9]{ds-cong} show that the Lucas congruences \eqref{eq:lucas}
in fact hold for the family of generalized Ap\'ery sequences
\begin{equation}
  A_{r, s} (n) = \sum_{k = 0}^n \binom{n}{k}^r \binom{n + k}{k}^s,
  \label{eq:A:x}
\end{equation}
with $r$ and $s$ positive integers. This family includes the sequences (a),
(b) from Table~\ref{tbl:sporadic2}, and the sequences ($\gamma$), $s_{10}$
from Table~\ref{tbl:sporadic3}. The purpose of this section and
Section~\ref{sec:lucas2} is to show that, in fact, all the Ap\'ery-like
sequences in Tables~\ref{tbl:sporadic2} and \ref{tbl:sporadic3} satisfy the
Lucas congruences \eqref{eq:lucas}. Using and extending the general framework
provided by R.~McIntosh \cite[Theorem~6]{mcintosh-lucas}, which we review
below, we are able to prove this claim for all of the sequences in the two
tables, with the exception of the two sequences $( \eta)$ and $s_{18}$, for
which we require a much finer analysis, which is given in
Section~\ref{sec:lucas2}.

\begin{theorem}
  \label{thm:lucas}Each of the sequences from Tables~\ref{tbl:sporadic2} and
  \ref{tbl:sporadic3} satisfies the Lucas congruences \eqref{eq:lucas}.
\end{theorem}

\begin{remark}
  The Lucas congruences \eqref{eq:lucas}, in general, do not extend to prime
  powers. However, it is shown in \cite{gessel-super}, and generalized in
  \cite{km-mod9}, that the Lucas congruences modulo $3$ for the Ap\'ery
  numbers extend to hold modulo $9$.
  
  On the other hand, numerical evidence suggests that all the Ap\'ery-like
  sequences from Tables~\ref{tbl:sporadic2} and \ref{tbl:sporadic3} in fact
  satisfy the Dwork congruences \eqref{eq:dwork}. While
  Theorem~\ref{thm:lucas} proves the case $r = 1$ of these congruences, it
  would be desirable to establish the corresponding congruences modulo higher
  powers of primes.
\end{remark}

Following \cite{mcintosh-lucas}, we say that a function $L :
\mathbb{Z}_{\geq 0}^2 \rightarrow \mathbb{Z}$ has the {\emph{double
Lucas property}} ({\textbf{DLP}}) if $L ( n, k) = 0$, for $k > n$, and if
\begin{equation}
  L ( n, k) \equiv L ( n_0, k_0) L ( n_1, k_1) \cdots L ( n_r, k_r)
  \pmod{p}, \label{eq:dlp}
\end{equation}
for every prime $p$. Here, as in \eqref{eq:lucas:bin}, $n_i$ and $k_i$ are the
$p$-adic digits of $n$ and $k$, respectively. Equation \eqref{eq:lucas:bin}
shows that the binomial coefficients $\binom{n}{k}$ are a {\textbf{DLP}}
function. More generally, it is shown in \cite[Theorem~6]{mcintosh-lucas}
that, for positive integers $r_0, r_1, \ldots, r_m$,
\begin{equation}
  L ( n, k) = \binom{n}{k}^{r_0} \binom{n + k}{k}^{r_1} \binom{n + 2
  k}{k}^{r_2} \cdots \binom{n + m k}{k}^{r_m} \label{eq:lucas:binx}
\end{equation}
is a {\textbf{DLP}} function. For instance, choosing the exponents as $r_i =
1$, we find that the multinomial coefficient
\begin{equation*}
  \binom{n + m k}{k, k, \ldots, k, n - k} = \frac{( n + m k) !}{k!^{m + 1} (
   n - k) !}
\end{equation*}
is a {\textbf{DLP}} function for any integer $m \geq 0$.

Suppose that $L ( n, k)$ is a {\textbf{DLP}} function and that $G ( n)$ and
$H ( n)$ are {\textbf{LP}} functions, that is, the sequences $G ( n)$ and $H
( n)$ satisfy the Lucas congruences \eqref{eq:lucas}. Then, as shown in
\cite[Theorem~5]{mcintosh-lucas},
\begin{equation}
  F ( n) = \sum_{k = 0}^n L ( n, k) G ( k) H ( n - k) \label{eq:dlp:sum}
\end{equation}
is an {\textbf{LP}} function. Note that \eqref{eq:lucas:binx} and
\eqref{eq:dlp:sum} combined are already sufficient to prove that the
generalized Ap\'ery sequences, defined in \eqref{eq:A:x}, satisfy Lucas
congruences. In order to apply this machinery more generally, and prove
Theorem~\ref{thm:lucas}, our next results extend the repertoire of
{\textbf{DLP}} functions. In fact, it turns out that we need a natural
extension of the Lucas property to the case of three variables. We say that a
function $M : \mathbb{Z}_{\geq 0}^3 \rightarrow \mathbb{Z}$ has the
{\emph{triple Lucas property}} ({\textbf{TLP}}) if $M ( n, k, j) = 0$, for
$j > n$, and if
\begin{equation}
  M ( n, k, j) \equiv M ( n_0, k_0, j_0) \cdots M ( n_r, k_r, j_r)
  \pmod{p}, \label{eq:tlp}
\end{equation}
for every prime $p$, where $n_i$, $k_i$ and $j_i$ are the $p$-adic digits of
$n$, $k$ and $j$, respectively. It is straightforward to prove the following
analog of \eqref{eq:dlp:sum} for {\textbf{TLP}} functions.

\begin{lemma}
  \label{lem:tlp:sum}If $M ( n, k, j)$ is a {\textbf{TLP}} function, then
  \begin{equation*}
    L ( n, k) = \sum_{j = 0}^n M ( n, k, j)
  \end{equation*}
  satisfies the double Lucas congruences \eqref{eq:dlp}. In particular, if $L
  ( n, k) = 0$, for $k > n$, then $L ( n, k)$ is a {\textbf{DLP}} function.
\end{lemma}

\begin{proof}
  Let $p$ be a prime. It is enough to show that, for any nonnegative integers
  $n_0, n', k_0, k'$ such that $n_0 < p$ and $k_0 < p$,
  \begin{equation}
    L ( n_0 + n' p, k_0 + k' p) \equiv L ( n_0, k_0) L ( n', k') \pmod{p} . \label{eq:lucas:L}
  \end{equation}
  Since the sum defining $L ( n, k)$ is naturally supported on $j \in \{ 0, 1,
  \ldots, n \}$, we may extend it over all $j \in \mathbb{Z}$. Modulo $p$, we
  have
  \begin{eqnarray*}
    L ( n, k) & = & \sum_{j \in \mathbb{Z}} M ( n, k, j)\\
    & = & \sum_{j_0 = 0}^{p - 1} \sum_{j' \in \mathbb{Z}} M ( n, k, j_0 + j'
    p)\\
    & \equiv & \sum_{j_0 \in \mathbb{Z}} \sum_{j' \in \mathbb{Z}} M ( n_0,
    k_0, j_0) M ( n', k', j')\\
    & = & L ( n_0, k_0) L ( n', k'),
  \end{eqnarray*}
  which is what we had to prove.
\end{proof}

\begin{lemma}
  \label{lem:tlp1}The function
  \begin{equation*}
    M ( n, k, j) = \binom{n}{j} \binom{k + j}{n}
  \end{equation*}
  is a {\textbf{TLP}} function.
\end{lemma}

\begin{proof}
  Clearly, $M ( n, k, j) = 0$, for $j > n$. In order to show that $M ( n, k,
  j)$ is a {\textbf{TLP}} function, we therefore need to show that, for any
  prime $p$,
  \begin{equation}
    M ( n_0 + n' p, k_0 + k' p, j_0 + j' p) \equiv M ( n_0, k_0, j_0) M ( n',
    k', j') \pmod{p}, \label{eq:lucas:M}
  \end{equation}
  provided that $0 \leq n_0, k_0, j_0 < p$ and $n', k', j' \geq 0$.
  Observe that in the case $j_0 > n_0$ both sides of the congruence
  \eqref{eq:lucas:M} vanish because of the Lucas congruences
  \eqref{eq:lucas:bin} for the binomial coefficients. We may therefore proceed
  under the assumption that $j_0 \leq n_0$.
  
  Writing $[ x^n] f ( x)$ for the coefficient of $x^n$ in the polynomial $f (
  x)$, we begin with the simple observation that
  \begin{equation*}
    \binom{k + j}{n} = [ x^n] ( 1 + x)^{k + j} .
  \end{equation*}
  Modulo $p$, we have
  \begin{equation*}
    ( 1 + x)^{k + j} = ( 1 + x)^{k_0 + j_0} ( 1 + x)^{( k' + j') p} \equiv (
     1 + x)^{k_0 + j_0} ( 1 + x^p)^{k' + j'} \pmod{p} .
  \end{equation*}
  Since $0 \leq k_0 + j_0 < 2 p$, extracting the coefficient of $x^n =
  x^{n_0} ( x^p)^{n'}$ from this product results in the congruence
  \begin{equation*}
    \binom{k + j}{n} \equiv \binom{k_0 + j_0}{n_0} \binom{k' + j'}{n'} +
     \binom{k_0 + j_0}{n_0 + p} \binom{k' + j'}{n' - 1} \pmod{p}
     .
  \end{equation*}
  Note that, under our assumption that $j_0 \leq n_0$, the second term on
  the right-hand side of this congruence vanishes (since $n_0 + p \geq
  j_0 + p > j_0 + k_0$). This, along with \eqref{eq:lucas:bin}, proves
  \eqref{eq:lucas:M}.
\end{proof}

\begin{corollary}
  \label{cor:dlp1}The function
  \begin{equation*}
    L ( n, k) = \binom{n}{k} \binom{2 k}{n}
  \end{equation*}
  is a {\textbf{DLP}} function.
\end{corollary}

\begin{proof}
  Set $j = k$ in Lemma~\ref{lem:tlp1}.
\end{proof}

\begin{lemma}
  \label{lem:dlp2}The function
  \begin{equation*}
    L ( n, k) = 3^{n - 3 k} \binom{n}{3 k} \frac{( 3 k) !}{k!^3}
  \end{equation*}
  is a {\textbf{DLP}} function.
\end{lemma}

\begin{proof}
  Let $p$ be a prime. As usual, we write $n = n_0 + n' p$ and $k = k_0 + k' p$
  where $0 \leq n_0 < p$ and $0 \leq k_0 < p$. In light of
  \eqref{eq:lucas:bin} and \eqref{eq:dlp:sum}, the simple observation
  \begin{equation}
    \binom{2 n}{n} = \sum_{k = 0}^n \binom{n}{k}^2, \label{eq:cenbin}
  \end{equation}
  demonstrates that the sequence of central binomial coefficients is an
  {\textbf{LP}} function. We claim that
  \begin{equation*}
    \frac{( 3 k) !}{k!^3} = \binom{3 k}{k} \binom{2 k}{k}
  \end{equation*}
  is an {\textbf{LP}} function as well. From the Lucas congruences for the
  central binomials, that is
  \begin{equation*}
    \binom{2 k}{k} \equiv \binom{2 k_0}{k_0} \binom{2 k'}{k'} \pmod{p},
  \end{equation*}
  we observe that $\binom{2 k}{k}$ is divisible by $p$ if $2 k_0 \geq p$.
  Hence, we only need to show the congruences
  \begin{equation}
    \frac{( 3 k) !}{k!^3} \equiv \frac{( 3 k_0) !}{k_0 !^3}  \frac{( 3 k')
    !}{k' !^3} \pmod{p} \label{eq:lucas:3k}
  \end{equation}
  under the assumption that $k_0 < p / 2$. Note that
  \begin{eqnarray*}
    \binom{3 k}{k} & = & [ x^k] ( 1 + x)^{3 k}\\
    & \equiv & [ x^{k_0} ( x^p)^{k'}] ( 1 + x)^{3 k_0} ( 1 + x^p)^{3 k'}
    \pmod{p}\\
    & = & \binom{3 k_0}{k_0} \binom{3 k'}{k'} + \binom{3 k_0}{k_0 + p}
    \binom{3 k'}{k' - 1} + \binom{3 k_0}{k_0 + 2 p} \binom{3 k'}{k' - 2} .
  \end{eqnarray*}
  In the case $k_0 < p / 2$, we have $k_0 + p > 3 k_0$, so that the last two
  terms on the right-hand side vanish. This proves \eqref{eq:lucas:3k}.
  
  Next, we claim that
  \begin{equation}
    \binom{n}{3 k} \frac{( 3 k) !}{k!^3} \equiv \binom{n_0}{3 k_0} \frac{( 3
    k_0) !}{k_0 !^3} \binom{n'}{3 k'} \frac{( 3 k') !}{k' !^3} \pmod{p} . \label{eq:lucas:3kb}
  \end{equation}
  By congruence \eqref{eq:lucas:3k}, both sides vanish modulo $p$ if $3 k_0
  \geq p$. On the other hand, if $3 k_0 < p$, then the usual argument
  shows that
  \begin{equation*}
    \binom{n}{3 k} \equiv [ x^{3 k_0} ( x^p)^{3 k'}] ( 1 + x)^{n_0} ( 1 +
     x^p)^{n'} = \binom{n_0}{3 k_0} \binom{n'}{3 k'} \pmod{p} .
  \end{equation*}
  In combination with \eqref{eq:lucas:3k}, this proves \eqref{eq:lucas:3kb}.
  
  Finally, the congruences $L ( n, k) \equiv L ( n_0, k_0) L ( n', k')$, that
  is
  \begin{equation}
    3^{n - 3 k} \binom{n}{3 k} \frac{( 3 k) !}{k!^3} \equiv 3^{n_0 - 3 k_0}
    \binom{n_0}{3 k_0} \frac{( 3 k_0) !}{k_0 !^3} 3^{n' - 3 k'} \binom{n'}{3
    k'} \frac{( 3 k') !}{k' !^3} \pmod{p},
    \label{eq:lucas:3kb3}
  \end{equation}
  follow from Fermat's little theorem and the fact that both sides vanish if
  $3 k_0 > n_0$ or $3 k' > n'$.
\end{proof}

We are now in a comfortable position to prove Theorem~\ref{thm:lucas} for all
but two of the sporadic Ap\'ery-like sequences. To show that sequences
$(\eta)$ and $s_{18}$ satisfy Lucas congruences as well requires considerable
additional effort, and the corresponding proofs are given in
Section~\ref{sec:lucas2}.

\begin{proof}[Proof of Theorem~\ref{thm:lucas}]
  Recall from \eqref{eq:cenbin} that the
  sequence of central binomial coefficients is an {\textbf{LP}} function.
  Further armed with \eqref{eq:lucas:binx} as well as Corollary~\ref{cor:dlp1}
  and Lemma~\ref{lem:dlp2}, the claimed Lucas congruences for the sequences $(
  \text{a})$, $( \text{b})$, $( \text{c})$, $( \text{d})$, $( \text{f})$, $(
  \alpha)$, $( \epsilon)$, $( \gamma)$, $s_{10}$, $s_7$ follow from
  \eqref{eq:dlp:sum}. It remains to consider the sequences $( \text{g})$, $(
  \delta)$, $( \zeta)$ as well as $( \eta)$ and $s_{18}$.
  
  Sequence $( \text{g})$ can be written as
  \begin{equation*}
    A_g ( n) = \sum_{k = 0}^n (- 1)^k 8^{n - k} \binom{n}{k} F ( k),
  \end{equation*}
  where $F ( k) = \sum_{l = 0}^k \binom{k}{l}^3$ are the Franel numbers
  (sequence $( \text{a})$), which we already know to be an {\textbf{LP}}
  function. As a consequence of Fermat's little theorem, the sequence $a^n$ is
  an {\textbf{LP}} function for any integer $a$. Hence, equation
  \eqref{eq:dlp:sum} applies to show that $A_g ( n)$ is an {\textbf{LP}}
  function.
  
  In order to see that sequence $( \delta)$ satisfies the Lucas congruences as
  well, it suffices to observe that $L ( n, k) = \binom{n + k}{k}$ is almost a
  {\textbf{DLP}} function, that is, it satisfies the congruences $(
  \ref{eq:dlp})$ but does not vanish for $k > n$. This is enough to conclude
  from Lemma~\ref{lem:dlp2} that
  \begin{equation*}
    L ( n, k) = 3^{n - 3 k} \binom{n}{3 k} \binom{n + k}{k} \frac{( 3 k)
     !}{k!^3}
  \end{equation*}
  is a {\textbf{DLP}} function. Since this is the summand of sequence $(
  \delta)$, the desired Lucas congruences again follow from
  \eqref{eq:dlp:sum}.
  
  On the other hand, for sequence $( \zeta)$, we observe that
  \begin{equation*}
    L ( n, k) = \sum_{j = 0}^n \binom{n}{j} \binom{k}{j} \binom{k + j}{n}
  \end{equation*}
  satisfies the congruences $( \ref{eq:dlp})$ by Lemma~\ref{lem:tlp:sum}
  because the summand is a {\textbf{TLP}} function by Lemma~\ref{lem:tlp1}.
  Hence, $\binom{n}{k}^2 L ( n, k)$ is a {\textbf{DLP}} function. Writing
  sequence $( \zeta)$ as
  \begin{equation*}
    A_{\zeta} ( n) = \sum_{k = 0}^n \binom{n}{k}^2 L ( n, k),
  \end{equation*}
  the claimed congruences once more follow from \eqref{eq:dlp:sum}.
\end{proof}

\section{Proofs for the two remaining sequences}\label{sec:lucas2}

The proof of the Lucas congruences in the previous section does not readily
extend to the sequences $( \eta)$ and $s_{18}$ from Table~\ref{tbl:sporadic3},
because, in contrast to the other cases, the known binomial sums for these
sequences do not have the property that their summands satisfy the double
Lucas property. Let us first note that the binomial sums for $s_{18}$ and
sequence $(\eta)$, given in $( \ref{eq:s18})$ and Table~\ref{tbl:sporadic3},
can be simplified at the expense of working with binomial coefficients with
negative entries. Namely, we have
\begin{equation}
  s_{18} (n) = \sum^n_{k = 0} (- 1)^k \binom{n}{k} \binom{2 k}{k} \binom{2 (n
  - k)}{n - k} \binom{2 n - 3 k}{n} \label{eq:s18:s}
\end{equation}
and
\begin{equation}
  A_{\eta} (n) = \sum_{k = 0}^n (- 1)^k \binom{n}{k}^3 \binom{4 n - 5 k}{3 n},
  \label{eq:eta:s}
\end{equation}
where, as usual, for any integer $m \geq 0$ and any number $x$, we define
\begin{equation*}
  \binom{x}{m} = \frac{x (x - 1) \cdots (x - m + 1)}{m!} .
\end{equation*}
For instance, the equivalence between \eqref{eq:s18} and \eqref{eq:s18:s} is a
simple consequence of the fact that, for integers $n \geq 0$ and $l = n -
k$,
\begin{equation}
  (- 1)^k \binom{2 n - 3 k}{n} = (- 1)^{k + n} \binom{- n + 3 k - 1}{n} = (-
  1)^l \binom{2 n - 3 l - 1}{n} . \label{eq:binom:neg23}
\end{equation}
For the first equality, we used that, for integers $b \geq 0$,
\begin{align}
  \binom{a}{b} & = \frac{a (a - 1) \cdots (a - b + 1)}{b!} \nonumber\\
  & = (- 1)^b  \frac{(- a) (- a + 1) \cdots (- a + b - 1)}{b!} = ( - 1)^b
  \binom{- a + b - 1}{b} .  \label{eq:binom:neg}
\end{align}
The following result generalizes the Lucas congruences for the sequence
$s_{18} (n)$.

\begin{theorem}
  \label{thm:s18:x:lucas}Suppose that $B (n, k)$ is a {\textbf{DLP}}
  function with the property that $B (n, k) = B (n, n - k)$. Then, the
  sequence
  \begin{equation*}
    A (n) = \sum^n_{k = 0} (- 1)^k B (n, k) \binom{2 n - 3 k}{n}
  \end{equation*}
  is an {\textbf{LP}} function, that is, $A ( n)$ satisfy the Lucas
  congruences \eqref{eq:lucas}.
\end{theorem}

\begin{proof}
  Let $p$ be a prime and let $n \geq 0$ be an integer. Write $n = n_0 +
  n' p$ and $k = k_0 + k' p$, where $0 \leq n_0  < p$ and $0
  \leq k_0 < p$ and $n', k'$ are nonnegative integers. We have to show
  that
  \begin{equation}
    A (n) \equiv A (n_0) A (n') \pmod{p} .
    \label{eq:s18:lucas}
  \end{equation}
  In the sequel, we denote
  \begin{equation*}
    C (n, k) = (- 1)^k B (n, k) \binom{2 n - 3 k}{n} .
  \end{equation*}
  For $k_0 \leq n_0 / 3$, we have $2 n_0 - 3 k_0 \geq n_0 \geq
  0$ and $2 n_0 - 3 k_0 \leq 2 n_0 < n_0 + p$. Hence, by the usual
  argument, we have
  \begin{eqnarray*}
    \binom{2 n - 3 k}{n} & \equiv & [ x^{n_0} ( x^p)^{n'}] ( 1 + x)^{2 n_0 - 3
    k_0} ( 1 + x^p)^{2 n' - 3 k'} \pmod{p}\\
    & \equiv & \binom{2 n_0 - 3 k_0}{n_0} \binom{2 n' - 3 k'}{n'}
    \pmod{p} .
  \end{eqnarray*}
  Hence, we find that, when $k_0 \leq n_0 / 3$,
  \begin{equation}
    C (n, k) \equiv C (n_0, k_0) C (n', k') \pmod{p} .
    \label{eq:s18:Cmod1}
  \end{equation}
  For $n_0 / 3 < k_0 < 2 n_0 / 3$, we have $n_0 > 2 n_0 - 3 k_0 > 0 $.
  By the same argument as above, we find that
  \begin{equation}
    \binom{2 n - 3 k}{n} \equiv 0 \pmod{p},
    \label{eq:s18:Cmod2}
  \end{equation}
  and hence $C (n, k) \equiv C (n_0, k_0) \equiv 0$ modulo $p$.
  
  Finally, consider the case $n_0 \geq 1$ and $2 n_0 / 3 \leq k_0
  \leq n_0$. In that case, $- p < - n_0 \leq 2 n_0 - 3 k_0 \leq
  0$ or, equivalently, $0 < 2 n_0 - 3 k_0 + p \leq p$. Hence, we have,
  modulo $p$,
  \begin{eqnarray}
    \binom{2 n - 3 k}{n} & \equiv & [ x^{n_0} ( x^p)^{n'}] ( 1 + x)^{2 n_0 - 3
    k_0 + p} ( 1 + x^p)^{2 n' - 3 k' - 1} \nonumber\\
    & \equiv & \binom{2 n_0 - 3 k_0 + p}{n_0} \binom{2 n' - 3 k' - 1}{n'}
    \nonumber\\
    & \equiv & \binom{2 n_0 - 3 k_0}{n_0} \binom{2 n' - 3 k' - 1}{n'}, 
    \label{eq:bin23:3}
  \end{eqnarray}
  because, for any integers $A, B$ and $m$ such that $0 \leq m < p$,
  \begin{eqnarray}
    \binom{A + B p}{m} & = & \frac{1}{m!}  (A + B p) (A + B p - 1) \cdots (A +
    B p - m + 1) \nonumber\\
    & \equiv & \frac{1}{m!} A (A - 1) \cdots (A - m + 1) = \binom{A}{m}
    \pmod{p} .  \label{eq:binom:cancelp}
  \end{eqnarray}
  Set $l' = n' - k'$. Applying \eqref{eq:binom:neg23} to the second binomial
  factor in \eqref{eq:bin23:3}, we find that
  \begin{eqnarray*}
    \binom{2 n - 3 k}{n} & \equiv & (- 1)^{n'} \binom{2 n_0 - 3 k_0}{n_0}
    \binom{2 n' - 3 l'}{n'} \pmod{p} .
  \end{eqnarray*}
  In combination with the assumed symmetry of $B (n, k)$, we therefore have
  that, when $n_0 \geq 1$ and $2 n_0 / 3 \leq k_0 \leq n_0$,
  \begin{equation}
    C (n, k) \equiv C (n_0, k_0) C (n', n' - k') \pmod{p} .
    \label{eq:s18:Cmod3}
  \end{equation}
  We are now ready to combine all cases. First, suppose that $n_0 \geq
  1$. Noting that $k \leq n / 3$ implies $k' \leq n' / 3$, and using
  \eqref{eq:s18:Cmod1}, \eqref{eq:s18:Cmod2} and \eqref{eq:s18:Cmod3}, we
  conclude that, modulo $p$,
  \begin{align*}
    A (n) & = \sum^{p - 1}_{k_0 = 0} \sum_{k' = 0}^{n'} C (n, k) \equiv
    \sum^{n_0}_{k_0 = 0} \sum_{k' = 0}^{n'} C (n, k)\\
    & \equiv \sum^{\lfloor n_0 / 3 \rfloor}_{k_0 = 0} \sum_{k' = 0}^{n'} C
    (n, k) + \sum^{n_0}_{k_0 = \lceil 2 n_0 / 3 \rceil} \sum_{k' = 0}^{n'} C
    (n, k)\\
    & \equiv \sum^{\lfloor n_0 / 3 \rfloor}_{k_0 = 0} C (n_0, k_0) \sum_{k'
    = 0}^{n'} C (n', k') + \sum^{n_0}_{k_0 = \lceil 2 n_0 / 3 \rceil} C (n_0,
    k_0) \sum_{k' = 0}^{n'} C (n', n' - k')\\
    & = \left[ \sum^{\lfloor n_0 / 3 \rfloor}_{k_0 = 0} C (n_0, k_0) +
    \sum^{n_0}_{k_0 = \lceil 2 n_0 / 3 \rceil} C (n_0, k_0) \right] \sum_{k' =
    0}^{n'} C (n', k')\\
    & = A (n_0) A (n'),
  \end{align*}
  which is what we wanted to prove. The case $n_0 = 0$ is simpler, and we only
  have to use \eqref{eq:s18:Cmod1} to again conclude that \eqref{eq:s18:lucas}
  holds.
\end{proof}

\begin{corollary}
  \label{cor:lucas:s18}The sequence $s_{18} (n)$ satisfies the Lucas
  congruences \eqref{eq:lucas}.
\end{corollary}

\begin{proof}
  Recall from the discussion in Section~\ref{sec:lucas} that
  \begin{equation*}
    B (n, k) = \binom{n}{k} \binom{2 k}{k} \binom{2 (n - k)}{n - k}
  \end{equation*}
  is a {\textbf{DLP}} function. Obviously, $B (n, k) = B (n, n - k)$. Hence,
  Theorem~\ref{thm:s18:x:lucas} applies to show that $s_{18} (n)$, in the form
  \eqref{eq:s18:s} satisfies the Lucas congruences \eqref{eq:lucas}.
\end{proof}

Next, we prove that the sequence $(\eta)$, which corresponds to the choice $a
= 3$ in Theorem~\ref{thm:lucas:eta}, satisfies Lucas congruences as well.

\begin{theorem}
  \label{thm:lucas:eta}Let $a \in \{1, 3\}$. Then, the sequence
  \begin{equation}
    A (n) = \sum^n_{k = 0} (- 1)^k \binom{n}{k}^a \binom{4 n - 5 k}{3 n}
    \label{eq:eta:a}
  \end{equation}
  is an {\textbf{LP}} function, that is, $A ( n)$ satisfy the Lucas
  congruences \eqref{eq:lucas}.
\end{theorem}

\begin{proof}
  Let $p$ be a prime and let $n \geq 0$ be an integer. As in the proof of
  Theorem~\ref{thm:s18:x:lucas}, we write $n = n_0 + n' p$ and $k = k_0 + k'
  p$, where $0 \leq n_0  < p$ and $0 \leq k_0 < p$ and $n',
  k'$ are nonnegative integers. Again, we have to show that
  \begin{equation}
    A (n) \equiv A (n_0) A (n') \pmod{p} .
    \label{eq:eta:lucas}
  \end{equation}
  Throughout the proof, let $d = \lfloor 3 n_0 / p \rfloor$.
  
  If $k_0 \leq n_0 / 5$, then $4 n_0 - 5 k_0 \geq 3 n_0 \geq 0$
  and $4 n_0 - 5 k_0 \leq 4 n_0 < 3 n_0 + p$. Since $d = \lfloor 3 n_0 /
  p \rfloor$, we thus have $0 \leq 3 n_0 - d p < p$ and $0 \leq 4
  n_0 - 5 k_0 - d p < (3 n_0 - d p) + p$. Therefore, modulo $p$,
  \begin{eqnarray*}
    \binom{4 n - 5 k}{3 n} & \equiv & [ x^{3 n_0 - d p} ( x^p)^{3 n' + d}] ( 1
    + x)^{4 n_0 - 5 k_0 - d p} ( 1 + x^p)^{4 n' - 5 k' + d}\\
    & \equiv & \binom{4 n_0 - 5 k_0 - d p}{3 n_0 - d p} \binom{4 n' - 5 k' +
    d}{3 n' + d}\\
    & \equiv & \binom{4 n_0 - 5 k_0}{3 n_0} \binom{4 n' - 5 k' + d}{3 n' +
    d},
  \end{eqnarray*}
  where in the last step we used that, modulo $p$,
  \begin{equation}
    \binom{4 n_0 - 5 k_0 - d p}{3 n_0 - d p} = \binom{4 n_0 - 5 k_0 - d p}{n_0
    - 5 k_0} \equiv \binom{4 n_0 - 5 k_0}{n_0 - 5 k_0} = \binom{4 n_0 - 5
    k_0}{3 n_0}, \label{eq:eta:bin45d}
  \end{equation}
  which follows from \eqref{eq:binom:cancelp} because $0 \leq n_0 - 5 k_0
  < p$. In particular, we have
  \begin{align}
    &\sum^{\lfloor n_0 / 5 \rfloor}_{k_0 = 0} \sum_{k' = 0}^{n'} (- 1)^k
    \binom{n}{k}^a \binom{4 n - 5 k}{3 n} \nonumber\\
    \equiv{} & \sum^{\lfloor n_0 / 5 \rfloor}_{k_0 = 0} (- 1)^{k_0}
    \binom{n_0}{k_0}^a \binom{4 n_0 - 5 k_0}{3 n_0} \sum_{k' = 0}^{n'} (-
    1)^{k'} \binom{n'}{k'}^a \binom{4 n' - 5 k' + d}{3 n' + d}, 
    \label{eq:eta:sum1}
  \end{align}
  and we observe that, for $d \in \{0, 1\}$,
  \begin{equation}
    A (n) = \sum_{k = 0}^n (- 1)^k \binom{n}{k}^a \binom{4 n - 5 k + d}{3 n +
    d} . \label{eq:eta:A1}
  \end{equation}
  To see this, note that the the sum of the $k$-th and $(n - k)$-th term does
  not depend on the value of $d \in \{0, 1\}$. Indeed, using
  \eqref{eq:binom:neg}, Pascal's relation and \eqref{eq:binom:neg} again, we
  deduce that
  \begin{eqnarray*}
    &  & \binom{4 n - 5 k + 1}{3 n + 1} + (- 1)^n \binom{4 n - 5 (n - k) +
    1}{3 n + 1}\\
    & = & \binom{4 n - 5 k + 1}{3 n + 1} - \binom{4 n - 5 k - 1}{3 n + 1}\\
    & = & \left[ \binom{4 n - 5 k + 1}{3 n + 1} - \binom{4 n - 5 k}{3 n + 1}
    \right] + \left[ \binom{4 n - 5 k}{3 n + 1} - \binom{4 n - 5 k - 1}{3 n +
    1} \right]\\
    & = & \binom{4 n - 5 k}{3 n} + \binom{4 n - 5 k - 1}{3 n}\\
    & = & \binom{4 n - 5 k}{3 n} + (- 1)^n \binom{4 n - 5 (n - k)}{3 n} .
  \end{eqnarray*}
  Next, suppose that $n_0 \geq 1$ and $4 n_0 / 5 \leq k_0 \leq
  n_0$. In that case, $- p < - n_0 \leq 4 n_0 - 5 k_0 \leq 0$ or,
  equivalently, $0 < 4 n_0 - 5 k_0 + p \leq p$. Hence, we have, modulo
  $p$,
  \begin{eqnarray*}
    \binom{4 n - 5 k}{3 n} & \equiv & [ x^{3 n_0 - d p} ( x^p)^{3 n' + d}] ( 1
    + x)^{4 n_0 - 5 k_0 + p} ( 1 + x^p)^{4 n' - 5 k' - 1}\\
    & \equiv & \binom{4 n_0 - 5 k_0 + p}{3 n_0 - d p} \binom{4 n' - 5 k' -
    1}{3 n' + d} .
  \end{eqnarray*}
  We rewrite the first binomial factor as follows, applying first
  \eqref{eq:binom:neg} and then \eqref{eq:binom:cancelp} twice, to find that,
  with $l_0 = n_0 - k_0$, modulo $p$,
  \begin{eqnarray*}
    \binom{4 n_0 - 5 k_0 + p}{3 n_0 - d p} & = & (- 1)^{n_0 + d} \binom{4 n_0
    - 5 l_0 - (d + 1) p - 1}{3 n_0 - d p}\\
    & \equiv & (- 1)^{n_0 + d} \binom{4 n_0 - 5 l_0 - d p - 1}{3 n_0 - d p}\\
    & = & (- 1)^{n_0 + d} \binom{4 n_0 - 5 l_0 - d p - 1}{n_0 - 5 l_0 - 1}\\
    & \equiv & (- 1)^{n_0 + d} \binom{4 n_0 - 5 l_0 - 1}{n_0 - 5 l_0 - 1}\\
    & = & (- 1)^{n_0 + d} \binom{4 n_0 - 5 l_0 - 1}{3 n_0} .
  \end{eqnarray*}
  Here, we proceeded under the assumption that $n_0 - 5 l_0 > 0$. It is
  straightforward to check that the final congruence also holds when $n_0 = 5
  l_0$, because then the binomial coefficients vanish modulo $p$. We conclude
  that, when $n_0 \geq 1$ and $4 n_0 / 5 \leq k_0 \leq n_0$,
  \begin{equation*}
    (- 1)^k \binom{4 n - 5 k}{3 n} \equiv (- 1)^{l_0} \binom{4 n_0 - 5 l_0 -
     1}{3 n_0} (- 1)^{k' + d} \binom{4 n' - 5 k' - 1}{3 n' + d} \pmod{p} .
  \end{equation*}
  In particular, we have
  \begin{align}
    & \sum^{n_0}_{k_0 = \lceil 4 n_0 / 5 \rceil} \sum_{k' = 0}^{n'} (-
    1)^k \binom{n}{k}^a \binom{4 n - 5 k}{3 n} \nonumber\\
    \equiv{} & \sum^{n_0}_{k_0 = \lceil 4 n_0 / 5 \rceil} (- 1)^{l_0}
    \binom{n_0}{l_0}^a \binom{4 n_0 - 5 l_0 - 1}{3 n_0} \sum_{k' = 0}^{n'} (-
    1)^{k' + d} \binom{n'}{k'}^a \binom{4 n' - 5 k' - 1}{3 n' + d} \nonumber\\
    ={} & \sum^{\lfloor n_0 / 5 \rfloor}_{k_0 = 0} (- 1)^{k_0}
    \binom{n_0}{k_0}^a \binom{4 n_0 - 5 k_0 - 1}{3 n_0} \sum_{k' = 0}^{n'} (-
    1)^{k' + d} \binom{n'}{k'}^a \binom{4 n' - 5 k' - 1}{3 n' + d}, 
    \label{eq:eta:sum3}
  \end{align}
  and we observe that, for integers $d \geq 0$,
  \begin{equation*}
    \sum_{k = 0}^n (- 1)^{k + d} \binom{n}{k}^a \binom{4 n - 5 k - 1}{3 n +
     d} = \sum_{k = 0}^n (- 1)^k \binom{n}{k}^a \binom{4 n - 5 k + d}{3 n + d}
  \end{equation*}
  because, by \eqref{eq:binom:neg},
  \begin{equation*}
    (- 1)^k \binom{4 n - 5 k + d}{3 n + d} = (- 1)^{(n - k) + d} \binom{4 n -
     5 (n - k) - 1}{3 n + d} .
  \end{equation*}
  Therefore, we can combine \eqref{eq:eta:sum1} and \eqref{eq:eta:sum3} into
  \begin{eqnarray}
    &  & \sum^{n_0}_{\substack{
      k_0 = 0\\
      k_0 \leq n_0 / 5 \text{ or } k_0 \geq 4 n_0 / 5
    }} \sum_{k' = 0}^{n'} (- 1)^k \binom{n}{k}^a \binom{4 n - 5
    k}{3 n} \nonumber\\
    & \equiv & A (n_0) \sum_{k' = 0}^{n'} (- 1)^{k'} \binom{n'}{k'}^a
    \binom{4 n' - 5 k' + d}{3 n' + d} \pmod{p}, 
    \label{eq:eta:sum13}
  \end{eqnarray}
  which holds for all $0 \leq n_0 < p$ (recall from the discussion at the
  beginning of this section that $A (n_0)$, like sequence $(\eta)$, can be
  represented as in Table~\ref{tbl:sporadic3}).
  
  On the other hand, suppose that $n_0 / 5 < k_0 < 4 n_0 / 5$. Set $f =
  \lfloor (4 n_0 - 5 k_0) / p \rfloor$. Since $0 < 4 n_0 - 5 k_0 < 3 n_0 < 3
  p$, we have $f \in \{0, 1, 2\}$. The usual arguments show that, modulo $p$,
  \begin{eqnarray}
    \binom{4 n - 5 k}{3 n} & \equiv & [ x^{3 n_0 - d p} ( x^p)^{3 n' + d}] ( 1
    + x)^{4 n_0 - 5 k_0 - f p} ( 1 + x^p)^{4 n' - 5 k' + f} \nonumber\\
    & \equiv & \binom{4 n_0 - 5 k_0 - f p}{3 n_0 - d p} \binom{4 n' - 5 k' +
    f}{3 n' + d} \nonumber\\
    & \equiv & \binom{4 n_0 - 5 k_0}{3 n_0 - d p} \binom{4 n' - 5 k' + f}{3
    n' + d} .  \label{eq:eta:bin2}
  \end{eqnarray}
  We are now in a position to begin piecing everything together. To do so, we
  consider individually the cases corresponding to the value of $d \in \{0, 1,
  2\}$.
  
  First, suppose $d = 0$ or $d = 1$. Congruence \eqref{eq:eta:sum13} coupled
  with \eqref{eq:eta:A1} implies that
  \begin{equation*}
    \sum^{n_0}_{\substack{
       k_0 = 0\\
       k_0 \leq n_0 / 5 \text{ or } k_0 \geq 4 n_0 / 5
     }} \sum_{k' = 0}^{n'} (- 1)^k \binom{n}{k}^a \binom{4 n - 5
     k}{3 n} \equiv A (n_0) A (n') \pmod{p} .
  \end{equation*}
  To conclude the desired congruence \eqref{eq:eta:lucas}, it therefore only
  remains to show that
  \begin{equation}
    \sum^{\lceil 4 n_0 / 5 \rceil - 1}_{k_0 = \lfloor n_0 / 5 \rfloor + 1}
    \sum_{k' = 0}^{n'} (- 1)^k \binom{n}{k}^a \binom{4 n - 5 k}{3 n} \equiv 0
    \pmod{p} . \label{eq:eta:sum2}
  \end{equation}
  This is easily seen in the case $d = 0$, because then each term of this sum
  vanishes modulo $p$. Equivalently, for $d = 0$, \eqref{eq:eta:bin2} vanishes
  whenever $n_0 / 5 < k_0 < 4 n_0 / 5$ (because $0 \leq 4 n_0 - 5 k_0 - f
  p \leq 4 n_0 - 5 k_0 < 3 n_0$). On the other hand, if $d = 1$, we claim
  that the sum \eqref{eq:eta:sum2} vanishes modulo $p$ because the terms
  corresponding to $(k_0, k')$ and $(k_0, n' - k')$ cancel each other. To see
  that, observe first that, for $d = 1$, \eqref{eq:eta:bin2} vanishes whenever
  $n_0 / 5 < k_0 < 4 n_0 / 5$ and $f = \lfloor (4 n_0 - 5 k_0) / p \rfloor
  \neq 0$ (because $0 \leq 4 n_0 - 5 k_0 - f p \leq 4 n_0 - 5 k_0 -
  p < 3 n_0 - p$ if $f \in \{1, 2\}$). Therefore, for the term corresponding
  to $(k_0, k')$,
  \begin{equation*}
    (- 1)^k \binom{4 n - 5 k}{3 n} \equiv (- 1)^{k_0} \binom{4 n_0 - 5 k_0}{3
     n_0 - p} (- 1)^{k'} \binom{4 n' - 5 k'}{3 n' + 1} \pmod{p},
  \end{equation*}
  while, for the term corresponding to $(k_0, n' - k')$ with $j = k_0 + (n' -
  k') p$,
  \begin{eqnarray*}
    (- 1)^j \binom{4 n - 5 j}{3 n} & \equiv & (- 1)^{k_0} \binom{4 n_0 - 5
    k_0}{3 n_0 - p} (- 1)^{n' - k'} \binom{4 n' - 5 (n' - k')}{3 n' + 1}\\
    & \equiv & (- 1)^{k_0} \binom{4 n_0 - 5 k_0}{3 n_0 - p} (- 1)^{k' + 1}
    \binom{4 n' - 5 k'}{3 n' + 1}\\
    & \equiv & - (- 1)^k \binom{4 n - 5 k}{3 n} \pmod{p},
  \end{eqnarray*}
  where we applied \eqref{eq:binom:neg} for the second congruence. It is now
  immediate to see that the sum \eqref{eq:eta:sum2} indeed vanishes modulo $p$
  for $d = 1$.
  
  It remains to prove the Lucas congruences \eqref{eq:eta:lucas} in the case
  $d = 2$. Using \eqref{eq:eta:sum13}, we have
  \begin{eqnarray*}
    A (n) & \equiv & A (n_0) \sum_{k' = 0}^{n'} (- 1)^{k'} \binom{n'}{k'}^a
    \binom{4 n' - 5 k' + 2}{3 n' + 2} + M \pmod{p},
  \end{eqnarray*}
  where
  \begin{equation*}
    M \assign \sum^{\lceil 4 n_0 / 5 \rceil - 1}_{k_0 = \lfloor n_0 / 5
     \rfloor + 1} \sum_{k' = 0}^{n'} (- 1)^k \binom{n}{k}^a \binom{4 n - 5
     k}{3 n} .
  \end{equation*}
  Combining this congruence with the identity
  \begin{equation*}
    A (n) = \sum_{k = 0}^n (- 1)^k \binom{n}{k}^a \left[ \binom{4 n - 5 k +
     2}{3 n + 2} - \binom{4 n - 5 k}{3 n + 2} \right],
  \end{equation*}
  which can be deduced along the same lines as \eqref{eq:eta:A1}, we find that
  \begin{equation}
    A (n) \equiv A (n_0) A (n') + A (n_0) \sum_{k' = 0}^{n'} (- 1)^{k'}
    \binom{n'}{k'}^a \binom{4 n' - 5 k'}{3 n' + 2} + M \pmod{p} . \label{eq:eta:A:d2}
  \end{equation}
  We have, by \eqref{eq:eta:bin2}, modulo $p$,
  \begin{align*}
    M & \equiv \sum^{\lceil 4 n_0 / 5 \rceil - 1}_{k_0 = \lfloor n_0 / 5
    \rfloor + 1} (- 1)^{k_0} \binom{n_0}{k_0}^a \binom{4 n_0 - 5 k_0}{3 n_0 -
    2 p} \sum_{k' = 0}^{n'} (- 1)^{k'} \binom{n'}{k'}^a \binom{4 n' - 5 k' +
    f}{3 n' + 2}\\
    & \equiv \sum^{\lceil 4 n_0 / 5 \rceil - 1}_{k_0 = \lfloor n_0 / 5
    \rfloor + 1} (- 1)^{k_0} \binom{n_0}{k_0}^a \binom{4 n_0 - 5 k_0}{3 n_0 -
    2 p} \sum_{k' = 0}^{n'} (- 1)^{k'} \binom{n'}{k'}^a \binom{4 n' - 5 k'}{3
    n' + 2},
  \end{align*}
  where the last congruence is a consequence of the identity
  \begin{equation*}
    \sum_{k = 0}^n (- 1)^k \binom{n}{k}^a \binom{4 n - 5 k + 1}{3 n + 2} =
     \sum_{k = 0}^n (- 1)^k \binom{n}{k}^a \binom{4 n - 5 k}{3 n + 2}
  \end{equation*}
  (which follows from \eqref{eq:binom:neg} and replacing $k$ with $n - k$) and
  the fact that \eqref{eq:eta:bin2} vanishes for $n_0 / 5 < k_0 < 4 n_0 / 5$
  if $f = 2$. Using this value of $M$ in \eqref{eq:eta:A:d2}, we find that the
  desired Lucas congruence \eqref{eq:eta:lucas} follows, if we can show that
  \begin{equation}
    A (n_0) + \sum^{\lceil 4 n_0 / 5 \rceil - 1}_{k_0 = \lfloor n_0 / 5
    \rfloor + 1} (- 1)^{k_0} \binom{n_0}{k_0}^a \binom{4 n_0 - 5 k_0}{3 n_0 -
    2 p} \equiv 0 \pmod{p} . \label{eq:eta:M0}
  \end{equation}
  Note that, if $k_0 \leq n_0 / 5$, then, by \eqref{eq:binom:cancelp} and
  \eqref{eq:eta:bin45d},
  \begin{equation}
    \binom{4 n_0 - 5 k_0}{3 n_0 - 2 p} \equiv \binom{4 n_0 - 5 k_0 - 2 p}{3
    n_0 - 2 p} \equiv \binom{4 n_0 - 5 k_0}{3 n_0} \pmod{p}
    . \label{eq:eta:bin45d2}
  \end{equation}
  A similar argument, combined with \eqref{eq:binom:neg}, shows that the
  congruence \eqref{eq:eta:bin45d2} also holds if $k_0 \geq 4 n_0 / 5$.
  We therefore find that \eqref{eq:eta:M0} is equivalent to
  \begin{equation*}
    \sum^{n_0}_{k_0 = 0} (- 1)^{k_0} \binom{n_0}{k_0}^a \binom{4 n_0 - 5
     k_0}{3 n_0 - 2 p} \equiv 0 \pmod{p} .
  \end{equation*}
  The next lemma proves that this congruence indeed holds provided that $a \in
  \{1, 3\}$.
\end{proof}

\begin{lemma}
  \label{lem:lucas:eta:0}Let $p$ be a prime, and $a \in \{1, 2, 3\}$. Then we
  have, for all $n$ such that $2 p / 3 \leq n < p$,
  \begin{equation*}
    \sum_{k = 0}^n (- 1)^{a k} \binom{n}{k}^a \binom{4 n - 5 k}{3 n - 2 p}
     \equiv 0 \pmod{p} .
  \end{equation*}
\end{lemma}

\begin{proof}
  To prove these congruences we employ N.~Calkin's technique
  \cite{calkin-bin} for proving similar divisibility results for sums of
  powers of binomials \eqref{eq:binpowersums}. Denoting $r = p - n$, we have,
  by \eqref{eq:binom:neg} and \eqref{eq:binom:cancelp},
  \begin{align*}
    \sum_{k = 0}^n (- 1)^{a k} \binom{n}{k}^a \binom{4 n - 5 k}{3 n - 2 p} & =
    \sum_{k = 0}^{p - r} (- 1)^{a k} \binom{p - r}{k}^a \binom{4 p - 4 r - 5
    k}{p - 3 r}\\
    & = \sum_{k = 0}^{p - r} \binom{k - p + r - 1}{k}^a \binom{4 p - 4 r -
    5 k}{p - 3 r}\\
    & \equiv \sum_{k = 0}^{p - r} \binom{k + r - 1}{k}^a \binom{4 p - 4 r -
    5 k}{p - 3 r} \pmod{p} .
  \end{align*}
  Clearly,
  \begin{equation}
    \binom{k + r - 1}{k} = \frac{(k + 1) (k + 2) \cdots (k + r - 1)}{(r - 1)
    !} = \frac{(k + 1)_{r - 1}}{(r - 1) !}, \label{eq:binpoch}
  \end{equation}
  where $(x)_k = x (x + 1) \cdots (x + k - 1)$ denotes the Pochhammer symbol
  (in particular, $(x)_0 = 1$). Likewise,
  \begin{equation*}
    \binom{4 p - 4 r - 5 k}{p - 3 r} = \frac{(3 p - r - 5 k + 1)_{p - 3
     r}}{(p - 3 r) !}
  \end{equation*}
  Since $(r - 1) !$ and $(p - 3 r) !$ are not divisible by $p$, we have to
  show that
  \begin{equation}
    \sum_{k = 0}^{p - r} (k + 1)_{r - 1}^a (3 p - r - 5 k + 1)_{p - 3 r}
    \equiv 0 \pmod{p} . \label{eq:etad:sump}
  \end{equation}
  Since the polynomials $(x)_k, (x)_{k - 1}, \ldots, (x)_0$ form an integer
  basis for the space of all polynomials with integer coefficients and degree
  at most $k$, there exist integers $c_0, c_1, \ldots, c_N$ with $N = (a - 1)
  (r - 1) + p - 3 r$ so that
  \begin{equation*}
    (k + 1)_{r - 1}^{a - 1} (3 p - r - 5 k + 1)_{p - 3 r} = \sum_{j = 0}^N
     c_j (k + r)_j .
  \end{equation*}
  Then the left-hand side of \eqref{eq:etad:sump} becomes
  \begin{eqnarray}
    \sum_{k = 0}^{p - r} (k + 1)_{r - 1} \sum_{j = 0}^N c_j (k + r)_j & = &
    \sum_{j = 0}^N c_j \sum_{k = 0}^{p - r} (k + 1)_{r - 1} (k + r)_j
    \nonumber\\
    & = & \sum_{j = 0}^N c_j \sum_{k = 0}^{p - r} (k + 1)_{r + j - 1}
    \nonumber\\
    & = & \sum_{j = 0}^N c_j  \frac{(p - r + 1)_{r + j}}{r + j}, 
    \label{eq:etad:sump2}
  \end{eqnarray}
  where we used
  \begin{equation*}
    (x)_k - (x - 1)_k = k (x)_{k - 1}
  \end{equation*}
  to evaluate
  \begin{equation*}
    \sum_{k = 0}^{p - r} (k + 1)_{r + j - 1} = \sum_{k = 0}^{p - r} \frac{(k
     + 1)_{r + j} - (k)_{r + j}}{r + j} = \frac{(p - r + 1)_{r + j}}{r + j} .
  \end{equation*}
  The desired congruence therefore follows if we can show that
  \begin{equation}
    \frac{(p - r + 1)_{r + j}}{r + j} \equiv 0 \pmod{p}
    \label{eq:etad:div}
  \end{equation}
  for all $j = 0, 1, \ldots, N$. Since $r > 0$ and $j \geq 0$, the
  numerator $(p - r + 1)_{r + j}$ is always divisible by $p$. The congruences
  \eqref{eq:etad:div} thus follow if $r + j < p$ for all $j$, or,
  equivalently, $r + N < p$. Since
  \begin{equation*}
    r + N = (a - 1) (r - 1) + p - 2 r,
  \end{equation*}
  we have $r + N < p$ if and only if
  \begin{equation*}
    (a - 1) (r - 1) < 2 r.
  \end{equation*}
  Clearly, this inequality holds for all $r \geq 1$ if and only if $a
  \leq 3$.
\end{proof}

\begin{remark}
  Numerical evidence suggests that the values $a \in \{1, 3\}$ in
  Theorem~\ref{thm:lucas:eta} are the only choices for which the sequence
  \eqref{eq:eta:a} satisfies Lucas congruences. In light of
  Lemma~\ref{lem:lucas:eta:0}, it is natural to ask if there are additional
  values of $a$ and $\varepsilon$, for which the sequence
  \begin{equation*}
    \sum^n_{k = 0} (- 1)^{\varepsilon k} \binom{n}{k}^a \binom{4 n - 5 k}{3
     n}
  \end{equation*}
  satisfies Lucas congruences. Empirically, this does not appear to be the
  case. In particular, for $a = 2$ this sequence does not satisfy Lucas
  congruences for either $\varepsilon = 0$ or $\varepsilon = 1$.
\end{remark}

\section{Periodicity of residues}\label{sec:periodic}

The Ap\'ery numbers satisfy
\begin{equation}
  A ( n) \equiv (- 1)^n \pmod{3}, \label{eq:A:3}
\end{equation}
and so are periodic modulo $3$. As in the case of the congruences
\eqref{eq:A:8}, which show that the Ap\'ery numbers are also periodic modulo
$8$, the congruences \eqref{eq:A:3} were first conjectured in
\cite{ccc-apery} and then proven in \cite{gessel-super}. We say that a
sequence $C ( n)$ is {\emph{eventually periodic}} if there exists an integer
$M > 0$ such that $C ( n + M) = C ( n)$ for all sufficiently large $n$. An
initial numerical search suggests that each sporadic Ap\'ery-like sequence
listed in Tables~\ref{tbl:sporadic2} and \ref{tbl:sporadic3} can only be
eventually periodic modulo a prime $p$ if $p \leq 5$. As an application
of Theorem~\ref{thm:lucas}, we prove this claim next.

\begin{corollary}
  \label{cor:notperiodic}None of the sequences from Tables~\ref{tbl:sporadic2}
  and \ref{tbl:sporadic3} is eventually periodic modulo $p$ for any prime $p >
  5$.
\end{corollary}

\begin{proof}
  Gessel \cite{gessel-super} shows that, if a sequence $C (n)$ satisfies the
  Lucas congruences \eqref{eq:lucas} modulo $p$ and is eventually periodic
  modulo $p$, then $C (n) \equiv C (1)^n$ modulo $p$ for all $n = 0, 1,
  \ldots, p - 1$.
  
  For instance, let $C (n)$ be the Almkvist--Zudilin sequence ($\delta$).
  Then, $C (1) = 3$, $C (2) = 9$ and $C (3) = 3$. Suppose $C (n)$ was
  eventually periodic modulo $p$. Then $p$ has to divide $C (3) - C (1)^3 = -
  24$, which implies that $p \in \{2, 3\}$.
  
  In Table~\ref{tbl:periodic} we list, for each sequence, the primes dividing
  both $C (2) - C (1)^2$ and $C (3) - C (1)^3$. The fact, that all these
  primes are at most $5$, proves our claim.
\end{proof}

\begin{table}[h]\centering
  \setlength{\tabcolsep}{3pt}
  \begin{tabular}{|c|c|c|c|c|c|c|c|c|c|c|c|c|c|c|}
    \hline
    (a) & (b) & (c) & (d) & (f) & (g) & ($\delta$) & ($\eta$) & ($\alpha$) &
    ($\epsilon$) & ($\zeta$) & ($\gamma$) & ($s_7$) & ($s_{10}$) &
    ($s_{18}$)\\
    \hline
    $2, 3$ & $2, 5$ & $2, 3$ & $2$ & $2, 3$ & $2, 3$ & $2, 3$ & $2, 5$ & $2,
    3$ & $2, 3$ & $2, 3$ & $2, 3$ & $2$ & $2$ & $2, 3$\\
    \hline
  \end{tabular}
  \caption{\label{tbl:periodic}The primes dividing both $C (2) - C (1)^2$ and
  $C (3) - C (1)^3$, for each sequence $C (n)$ from Tables~\ref{tbl:sporadic2}
  and \ref{tbl:sporadic3}.}
\end{table}

In fact, as another simple consequence of Theorem~\ref{thm:lucas}, we observe
that the Ap\'ery-like sequences are in fact eventually periodic modulo each
of the primes listed in Table~\ref{tbl:periodic}.

\begin{corollary}
  \label{cor:periodic}Let $C (n)$ be any sequence from
  Tables~\ref{tbl:sporadic2} and \ref{tbl:sporadic3}.
  \begin{itemize}
    \item $C (n) \equiv C (1) \pmod{2}$ for all $n
    \geq 1$.
    
    \item $C (n) \equiv C (1) \pmod{3}$ for all $n
    \geq 1$ if $C (n)$ is one of $( c)$, $( f)$, $( g)$, $(\delta)$,
    $(\alpha)$, $(\epsilon)$, $(\zeta)$, $s_{18}$, and $C (n) \equiv ( - 1)^n
    \pmod{3}$ for all $n \geq 0$ if $C (n)$ is $( a)$
    or $(\gamma)$.
    
    \item $C (n) \equiv 3^n \pmod{5}$ for all $n \geq
    0$ if $C (n)$ is $( b)$, and $C (n) \equiv 0 \pmod{5}$
    for all $n \geq 1$ if $C (n)$ is $( \eta)$.
  \end{itemize}
\end{corollary}

\begin{proof}
  One can check that Table~\ref{tbl:periodic} does not change if we include
  only those primes $p$ such that $C (n) - C (1)^n$ is divisible by $p$ for
  all $n \in \{0, 1, 2, 3, 4\}$. For $n = 0$, this is trivial since $C ( 0) =
  1$. Therefore, in each of the cases considered here, we have
  \begin{equation*}
    C (n) \equiv C (1)^n \pmod{p}
  \end{equation*}
  for all $n \in \{0, 1, \ldots, p - 1\}$. For any $n \geq 0$, let $n =
  n_0 + n_1 p + \cdots + n_r p^r$ be the $p$-adic expansion of $n$. Then, by
  Theorem~\ref{thm:lucas}, we have
  \begin{eqnarray*}
    C (n) & \equiv & C (n_0) C (n_1) \cdots C (n_r) \pmod{p}\\
    & \equiv & C (1)^{n_0 + n_1 + \cdots + n_r} \pmod{p}\\
    & \equiv & C (1)^n \pmod{p} .
  \end{eqnarray*}
  For the final congruence we used Fermat's little theorem. All claimed
  congruences then follow from the specific initial values of $C (n)$ modulo
  $p$.
\end{proof}

More interestingly, the congruences \eqref{eq:A:8} show that the Ap\'ery
numbers (sequence $(\gamma)$) are periodic modulo $8$. We offer the following
corresponding result for the Almkvist--Zudilin sequence $(\delta)$.

\begin{theorem}
  \label{thm:AZ8}The Almkvist--Zudilin numbers
  \begin{equation*}
    Z (n) = \sum_{k = 0}^n (- 1)^k 3^{n - 3 k} \binom{n}{3 k} \binom{n +
     k}{n} \frac{(3 k) !}{k!^3}
  \end{equation*}
  satisfy the congruences
  \begin{equation}
    Z ( n) \equiv \left\{ \begin{array}{ll}
      1, & \text{if $n$ is even,}\\
      3, & \text{if $n$ is odd,}
    \end{array} \right. \pmod{8} . \label{eq:AZ:8}
  \end{equation}
\end{theorem}

\begin{proof}
  It is shown in \cite{s-apery} that the numbers $(- 1)^n Z (n)$ are the
  diagonal Taylor coefficients of the multivariate rational function
  \begin{equation}
    F (x_1, x_2, x_3, x_4) = \frac{1}{1 - ( x_1 + x_2 + x_3 + x_4) + 27 x_1
    x_2 x_3 x_4} . \label{eq:AZ:rat}
  \end{equation}
  That is, if
  \begin{equation*}
    F (x_1, x_2, x_3, x_4) = \sum_{n_1 = 0}^{\infty} \sum_{n_2 = 0}^{\infty}
     \sum_{n_3 = 0}^{\infty} \sum_{n_4 = 0}^{\infty} C (n_1, n_2, n_3, n_4)
     x_1^{n_1} x_2^{n_2} x_3^{n_3} x_4^{n_4}
  \end{equation*}
  is the Taylor expansion of the rational function $F$, then $Z (n) = (- 1)^n
  C (n, n, n, n)$.
  
  Given such a rational function as well as a reasonably small prime power
  $p^r$, Rowland and Yassawi \cite{ry-diag13} give an explicit algorithm for
  computing a finite state automaton, which produces the values of the
  diagonal coefficients modulo $p^r$. In the present case, this finite state
  automaton for the values $(- 1)^n Z (n)$ modulo $8$ turns out to be the same
  automaton as the one for the Ap\'ery numbers modulo $8$. Hence, the
  congruences \eqref{eq:AZ:8} follow from the congruences \eqref{eq:A:8}. We
  refer to \cite{ry-diag13} for details on finite state automata and the
  algorithm to construct them from a multivariate rational generating
  function.
\end{proof}

Empirically, Theorem~\ref{thm:AZ8} is the only other interesting set of
congruences, apart from the congruences \eqref{eq:A:8}, which demonstrates
that an Ap\'ery-like sequence is periodic modulo a prime power. More
precisely, numerical evidence suggests that none of the sequences in
Tables~\ref{tbl:sporadic2} and \ref{tbl:sporadic3} is eventually periodic
modulo $p^r$, for some $r > 1$, unless $p = 2$. Moreover, the only other
instances modulo a power of $2$ appear to be the following, less interesting,
ones: sequences $( \text{d})$ and $(\alpha)$ are eventually periodic modulo
$4$ because all their terms, except the first, are divisible by $4$; likewise,
sequences $(\varepsilon)$ and $s_7$ are eventually periodic modulo $8$ because
all their terms, except the first, are divisible by $8$. We do not attempt to
prove these claims here. We remark, however, that these claims can be
established by the approach used in the proof of Theorem~\ref{thm:AZ8},
provided that one is able to determine a computationally accessible analog of
\eqref{eq:AZ:rat} for the sequence at hand.

\section{Primes not dividing Ap\'ery-like numbers}\label{sec:primes}

Using the Lucas congruences proved in Theorem~\ref{thm:lucas}, it is
straightforward to verify whether or not a given prime divides some
Ap\'ery-like number.

\begin{example}
  \label{eg:A:7}The values of Ap\'ery numbers $A ( 0), A ( 1), \ldots, A (
  6)$ modulo $7$ are $1, 5, 3, 3, 3, 5, 1$. Since $7$ does not divide $A ( 0),
  A ( 1), \ldots, A ( 6)$, it follows from the Lucas congruences
  \eqref{eq:lucas} that $7$ does not divide any Ap\'ery number.
\end{example}

Arguing as in Example~\ref{eg:A:7}, one finds that the primes $2, 3, 7, 13,
23, 29, 43, 47, \ldots$ do not divide any Ap\'ery number $A ( n)$.
E.~Rowland and R.~Yassawi \cite{ry-diag13} pose the question whether there
are infinitely many such primes. Table~\ref{tbl:primes} records, for each
sporadic Ap\'ery-like sequence, the primes below $100$ which do not divide
any of its terms, and the last column gives the proportion of primes below
$10^4$ with this property. Each Ap\'ery-like sequence is specified by its
label from \cite{asz-clausen}, which is also used in
Tables~\ref{tbl:sporadic2} and \ref{tbl:sporadic3}. The alert reader will
notice that Cooper's sporadic sequences (the ones with $d \neq 0$ in
Table~\ref{tbl:sporadic3}) are missing from Table~\ref{tbl:primes}. That is
because these sequences turn out to be divisible by all primes. A more precise
result for these sequences is proved at the end of this section.

\begin{table}[h]\centering
  \begin{tabular}{|c|l|l|}
    \hline
    (a) & 3, 11, 17, 19, 43, 83, 89, 97 & 0.2994\\
    \hline
    (b) & 2, 5, 13, 17, 29, 37, 41, 61, 73, 89 & 0.2897\\
    \hline
    (c) & 2, 7, 13, 37, 61, 73 & 0.2962\\
    \hline
    (d) & 3, 11, 17, 19, 43, 59, 73, 83, 89 & 0.2815\\
    \hline
    (f) & 2, 5, 13, 17, 29, 37, 41, 61, 73, 97 & 0.2994\\
    \hline
    (g) & 5, 11, 29, 31, 59, 79 & 0.2929\\
    \hline
    ($\delta$) & 2, 5, 7, 11, 13, 19, 29, 41, 47, 61, 67, 71, 73, 89, 97 &
    0.6192\\
    \hline
    ($\eta$) & 2, 3, 17, 19, 23, 31, 47, 53, 61 & 0.2897\\
    \hline
    ($\alpha$) & 3, 5, 13, 17, 29, 31, 37, 41, 43, 47, 53, 59, 61, 67, 71, 83,
    89 & 0.5989\\
    \hline
    ($\epsilon$) & 3, 7, 13, 19, 23, 29, 31, 37, 43, 47, 61, 67, 73, 83, 89 &
    0.6037\\
    \hline
    ($\zeta$) & 2, 5, 7, 13, 17, 19, 29, 37, 43, 47, 59, 61, 67, 71, 83, 89 &
    0.6046\\
    \hline
    ($\gamma$) & 2, 3, 7, 13, 23, 29, 43, 47, 53, 67, 71, 79, 83, 89 &
    0.6168\\
    \hline
  \end{tabular}
  \caption{\label{tbl:primes}The primes below $100$ not dividing
  Ap\'ery-like numbers (sequence indicated in first column using the labels
  from \cite{asz-clausen}) as well as the proportion of primes (in the last
  column) below $10, 000$ not dividing any term}
\end{table}

Example~\ref{eg:A:7} shows that the first $7$ values of the Ap\'ery numbers
modulo $7$ are palindromic. Our next result, which was noticed by E.~Rowland,
shows that this is true for all primes.

\begin{lemma}
  \label{lem:A:palin}For any prime $p$, and integers $n$ such that $0
  \leq n < p$, the Ap\'ery numbers $A ( n)$ satisfy the congruence
  \begin{equation}
    A ( n) \equiv A ( p - 1 - n) \pmod{p} . \label{eq:A:palin}
  \end{equation}
\end{lemma}

\begin{proof}
  For $n$ such that $0 \leq n < p$, we employ \eqref{eq:binom:neg} and
  \eqref{eq:binom:cancelp} to arrive at
  \begin{eqnarray*}
    A ( p - 1 - n) & = & \sum_{k = 0}^{p - 1} \binom{p - 1 - n}{k}^2 \binom{p
    - 1 - n + k}{k}^2\\
    & \equiv & \sum_{k = 0}^{p - 1} \binom{n + k}{k}^2 \binom{n}{k}^2 = A (
    n) \pmod{p},
  \end{eqnarray*}
  as claimed.
\end{proof}

Theorem~\ref{thm:lucas} and Lemma~\ref{lem:A:palin}, considered together,
suggest that $e^{- 1 / 2} \approx 60.65\%$ of the primes do not divide any
Ap\'ery number. Indeed, let us make the empirical assumption that the values
$A ( n)$ modulo $p$, for $n = 0, 1, \ldots, ( p - 1) / 2$, are independent and
uniformly random. Since one of the values $A ( n)$ is congruent to $0$ modulo
$p$ with probability $1 / p$, it follows that the probability that $p$ does
not divide any of the $( p + 1) / 2$ first values is
\begin{equation}
  \left( 1 - \frac{1}{p} \right)^{( p + 1) / 2} . \label{eq:prop:p}
\end{equation}
By the Lucas congruences, shown in Theorem~\ref{thm:lucas}, and
Lemma~\ref{lem:A:palin}, $p$ does not divide any of the $( p + 1) / 2$ first
values if and only if $p$ does not divide any Ap\'ery number. In the limit
$p \rightarrowlim \infty$, the proportion \eqref{eq:prop:p} becomes $e^{- 1 /
2}$. Observe that this empirical prediction matches the numerical data in
Table~\ref{tbl:primes} rather well. We therefore arrive at the following
conjecture.

\begin{conjecture}
  \label{conj:prop:apery}The proportion of primes not dividing any Ap\'ery
  number $A ( n)$ is $e^{- 1 / 2}$.
\end{conjecture}

\begin{figure}[h]\centering
  \includegraphics[width=.95\textwidth]{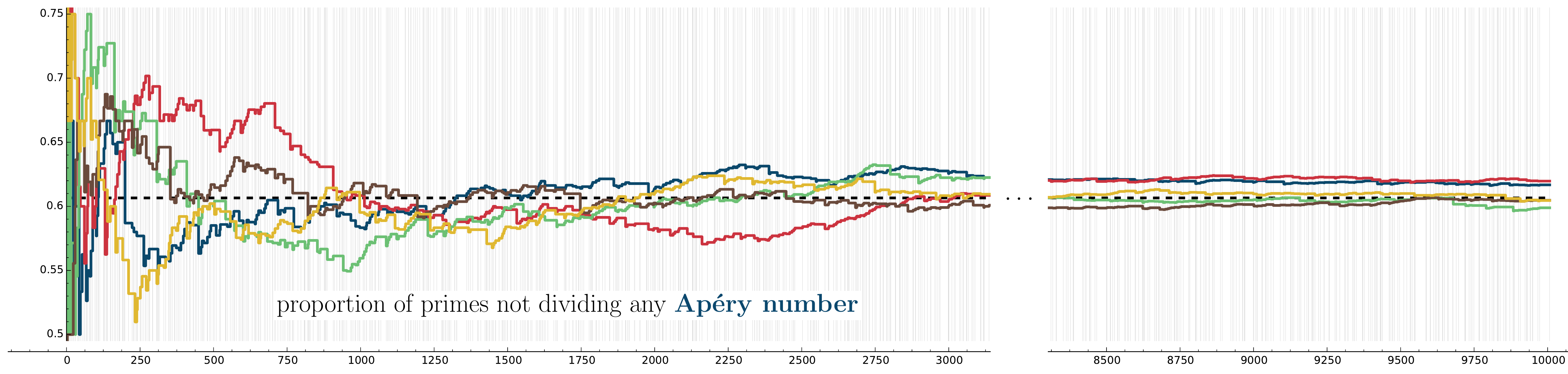}
  \caption{\label{fig:primes}Proportion of primes (up to $10, 000$) not
  dividing the sequences $( \delta)$, $( \alpha)$, $( \epsilon)$, $( \zeta)$,
  $( \gamma)$, with the dotted line indicating $e^{- 1 / 2}$. The Ap\'ery
  sequence is plotted in blue. (We thank Arian Daneshvar for producing this
  plot.)}
\end{figure}

While Lemma~\ref{lem:A:palin} does not hold for the other Ap\'ery-like
numbers $C ( n)$ from Tables~\ref{tbl:sporadic2} and \ref{tbl:sporadic3}, we
make the weaker observation that if a prime $p > 5$ divides $C ( n)$, where $0
\leq n < p$, then $p$ also divides $C ( p - 1 - n)$. We expect that this
empirical observation can be proven in the spirit of the proof of
Lemma~\ref{lem:A:palin}, but do not pursue this theme further. We only note
that it allows us to extend the heuristic leading to
Conjecture~\ref{conj:prop:apery} to the Ap\'ery-like sequences $( \delta)$,
$( \alpha)$, $( \epsilon)$, $( \zeta)$ from Table~\ref{tbl:sporadic3}. In
other words, we conjecture that, for each of these sequences, the proportion
of primes not dividing any of the terms is again $e^{- 1 / 2}$.
Figure~\ref{fig:primes} visualizes some numerical evidence for this
conjecture. On the other hand, for sequence $( \eta)$ as well as the sequences
from Table~\ref{tbl:sporadic2}, the proportion of primes not dividing any of
their terms appears to be about half of that, that is $e^{- 1 / 2} / 2 \approx
30.33\%$.

To explain this extra factor of $1 / 2$, we note that, for the Ap\'ery-like
numbers
\begin{equation}
  A_b ( n) = \sum_k \binom{n}{k}^2 \binom{n + k}{n}, \label{eq:Ab}
\end{equation}
Stienstra and Beukers \cite{sb-picardfuchs} proved that, modulo $p$,
\begin{equation}
  A_b \left( \frac{p - 1}{2} \right) \equiv \left\{ \begin{array}{ll}
    4 a^2 - 2 p, & \text{if $p = a^2 + b^2$, $a$ odd,}\\
    0, & \text{if $p \equiv 3 \pmod{4}$}
  \end{array} \right. \label{eq:Ab:half}
\end{equation}
(and conjectured that the congruence should hold modulo $p^2$, which was later
proved by Ahlgren and Ono \cite{ahlgren-ono-apery}). In particular,
congruence \eqref{eq:Ab:half} makes it explicit that every prime $p \equiv 3
\pmod{4}$ divides a term of this Ap\'ery-like sequence.
Note that, by a classical congruence of Gauss, the congruences
\eqref{eq:Ab:half} are equivalent, modulo $p$, to the congruences
\begin{equation}
  A_b \left( \left\lfloor \frac{p}{2} \right\rfloor \right) \equiv \left\{
  \begin{array}{ll}
    \binom{\lfloor p / 2 \rfloor}{\lfloor p / 4 \rfloor}^2, & \text{if $p
    \equiv 1 \pmod{4}$},\\
    0, & \text{otherwise},
  \end{array} \right. \label{eq:Ab:half:b}
\end{equation}
which are valid for any prime $p \neq 2$. The more general result in
\cite{sb-picardfuchs} also includes the cases $A_a$ and $A_c$. Similar
divisibility results appear to hold for the other Ap\'ery-like numbers from
Table~\ref{tbl:sporadic2}, and it would be interesting to make these explicit.

On the other hand, the extra factor of $1 / 2$ in case of sequence $( \eta)$
is explained by the following congruences, which resemble \eqref{eq:Ab:half:b}
remarkably well.

\begin{theorem}
  For any prime $p \neq 3$, we have that, modulo $p$,
  \begin{equation}
    A_{\eta} \left( \left\lfloor \frac{p}{3} \right\rfloor \right) \equiv
    \left\{ \begin{array}{ll}
      (- 1)^{\lfloor p / 5 \rfloor} \binom{\lfloor p / 3 \rfloor}{\lfloor p /
      15 \rfloor}^3, & \text{if $p \equiv 1, 2, 4, 8 \pmod{15}$},\\
      0, & \text{otherwise} .
    \end{array} \right. \label{eq:Aeta:third}
  \end{equation}
\end{theorem}

\begin{proof}
  Suppose that $p \equiv 2 \pmod{3}$, and write $p = 3 n +
  2$. The congruence \eqref{eq:Aeta:third} can be checked directly for $p = 2$
  and $p = 5$, and so we may assume $p > 5$ in the sequel. Applying
  \eqref{eq:binpoch} to the definition of sequence $( \eta)$ in
  Table~\ref{tbl:sporadic3}, we have
  \begin{align}
    A_{\eta} (n) & = \sum_{k = 0}^{\lfloor n / 5 \rfloor} (- 1)^k
    \binom{n}{k}^3 \left( \binom{4 n - 5 k - 1}{3 n} + \binom{4 n - 5 k}{3 n}
    \right) \nonumber\\
    & = \sum_{k = 0}^{\lfloor n / 5 \rfloor} (- 1)^k \binom{n}{k}^3 \left(
    \frac{( n - 5 k)_{3 n}}{( 3 n) !} + \frac{( n - 5 k + 1)_{3 n}}{( 3 n) !}
    \right) .  \label{eq:etad:sp}
  \end{align}
  Since $3 n = p - 2$ and $0 \leq k \leq n / 5$, the term
  \begin{equation}
    \frac{( n - 5 k)_{3 n}}{( 3 n) !} \label{eq:etad:poch}
  \end{equation}
  is always divisible by $p$, unless $n - 5 k \in \{ 1, 2 \}$ (because,
  otherwise, one of the $p - 2$ factors of $( n - 5 k)_{3 n}$ is divisible by
  $p$, while $( 3 n) !$ is not). Note that $n - 5 k = 1$ and $n - 5 k = 2$ are
  equivalent to $k = ( p - 5) / 15$ and $k = ( p - 8) / 15$, respectively.
  However, $( p - 5) / 15$ cannot be an integer (since $p \neq 5$). We thus
  find that \eqref{eq:etad:poch} vanishes modulo $p$ unless $p \equiv 8
  \pmod{15}$ and $k = \lfloor p / 15 \rfloor$, in which
  case \eqref{eq:etad:poch} is congruent to $- 1$ modulo $p$. Combined with
  the analogous discussion for the corresponding second term in
  \eqref{eq:etad:sp}, we conclude that
  \begin{equation*}
    \frac{( n - 5 k)_{3 n}}{( 3 n) !} + \frac{( n - 5 k + 1)_{3 n}}{( 3 n) !}
     \equiv \left\{ \begin{array}{ll}
       1, & \text{if $k = \lfloor p / 15 \rfloor$ and $p \equiv 2 \pmod{15}$},\\
       - 1, & \text{if $k = \lfloor p / 15 \rfloor$ and $p \equiv 8
       \pmod{15}$},\\
       0, & \text{otherwise} .
     \end{array} \right.
  \end{equation*}
  Applying this to the sum \eqref{eq:etad:sp} and combining the signs
  properly, we arrive at the congruences \eqref{eq:Aeta:third} when $p \equiv
  2 \pmod{3}$.
  
  The case $p \equiv 1 \pmod{3}$ is similar and a little
  bit simpler.
\end{proof}

In summary, we conjecture that the proportion of primes not dividing any term
of the Ap\'ery-like sequences in Tables~\ref{tbl:sporadic2} and
\ref{tbl:sporadic3} is as follows.

\begin{conjecture}
  \label{conj:prop:aperyx}{\tmdummy}
  
  \begin{itemize}
    \item Let $C ( n)$ be one of the sequences of Table~\ref{tbl:sporadic2} or
    sequence $( \eta)$. Then the proportion of primes not dividing any $C (
    n)$ is $\frac{1}{2} e^{- 1 / 2}$.
    
    \item Let $C ( n)$ be one of the sequences $( \delta)$, $( \alpha)$, $(
    \epsilon)$, $( \zeta)$, $( \gamma)$ from Table~\ref{tbl:sporadic3}. Then
    the proportion of primes not dividing any $C ( n)$ is $e^{- 1 / 2}$.
  \end{itemize}
\end{conjecture}

In stark contrast, Cooper's sporadic sequences $s_7$, $s_{10}$, $s_{18}$ from
Table~\ref{tbl:sporadic3} are divisible by all primes. Indeed, let $C ( n)$
denote any of these three sequences. Then,
\begin{equation*}
  C ( p - 1) \equiv 0 \pmod{p}
\end{equation*}
for all primes $p$. In fact, we can prove much more. For any given prime $p$,
the last quarter (or third) of the first $p$ terms of these sequences are
divisible by $p$. In the case of sequence $s_{10}$, the sum of fourth powers
of binomial coefficients, this is proved by N.~Calkin \cite{calkin-bin}.
Indeed, among other divisibility results on sums of powers of binomials,
Calkin shows that, for all integers $a \geq 0$, the sums
\begin{equation}
  \sum_{k = 0}^n \binom{n}{k}^{2 a} \label{eq:binpowersums}
\end{equation}
are divisible by all primes $p$ in the range
\begin{equation*}
  n < p < n + 1 + \frac{n}{2 a - 1} .
\end{equation*}
In particular, in the case $a = 2$, we conclude that $s_{10} ( n)$ is
divisible by all primes $p$ that satisfy $n < p < \frac{4 n}{3} + 1$.
Equivalently, we have
\begin{equation*}
  s_{10} ( p - j) \equiv 0 \pmod{p}
\end{equation*}
whenever $1 \leq j \leq ( p + 2) / 4$. Our final result proves the
same phenomenon for Cooper's sporadic sequences $s_7, s_{18}$. We note that in
each case, empirically, the bounds on $j$ cannot be improved (with the
expection of the case $p = 3$ for $s_{18}$; see Remark~\ref{rk:s18:mod3}).

\begin{theorem}
  \label{thm:cc}For any prime $p$, we have
  \begin{equation*}
    s_7 ( p - j) \equiv 0 \pmod{p}
  \end{equation*}
  whenever $1 \leq j \leq ( p + 1) / 3$, and
  \begin{equation*}
    s_{18} ( p - j) \equiv 0 \pmod{p}
  \end{equation*}
  whenever $1 \leq j \leq ( p + 2) / 4$.
\end{theorem}

\begin{proof}
  For the sequence $s_7$, we want to show
  \begin{equation*}
    \sum_{k = 0}^{p - j} \binom{p - j}{k}^2 \binom{p - j + k}{k} \binom{2
     k}{p - j} \equiv 0 \pmod{p},
  \end{equation*}
  for $1 \leq j \leq ( p + 1) / 3$. Note that for $2 k < p - j$ or
  $k > p - j$ the summand is already zero. Therefore, we assume that $p - j
  \geq k \geq ( p - j) / 2$. We write the summand as
  \begin{equation*}
    \binom{p - j}{k}^2 \binom{p - j + k}{k} \binom{2 k}{p - j} = \frac{( p -
     j + k) ! ( 2 k) !}{k!^3 ( p - j - k) !^2 ( 2 k - p + j) !},
  \end{equation*}
  and observe that the denominator is not divisible by $p$ if $j \geq 1$.
  On the other hand, the factorial $( p - j + k) !$ in the numerator is
  divisible by $p$ since
  \begin{equation*}
    p - j + k \geq p - j + \left\lceil \frac{p - j}{2} \right\rceil
     \geq p,
  \end{equation*}
  where we used $j \leq ( p + 1) / 3$ to verify the final inequality.
  Thus, for $1 \leq j \leq ( p + 1) / 3$, the congruences $s_7 ( p -
  j) \equiv 0$ hold modulo $p$, as claimed.
  
  We proceed similarly for $s_{18} ( p - j)$, which is given by
  \begin{equation*}
    \scaleeq{0.88}{\sum^{\lfloor ( p - j) / 3 \rfloor}_{k = 0} (- 1)^k \binom{p -
    j}{k} \binom{2 k}{k} \binom{2 (p - j - k)}{p - j - k} \left\{ \binom{2 (
    p - j) - 3 k - 1}{p - j} + \binom{2 ( p - j) - 3 k}{p - j} \right\},}
  \end{equation*}
  and, using \eqref{eq:binpoch}, write the summand as
  \begin{equation}
    \frac{( - 1)^k ( 2 k) ! ( 2 (p - j - k)) !}{k!^3 ( p - j - k) !^3}  ( p -
    j - 3 k + 1)_{p - j - 1} ( 3 p - 3 j - 6 k) . \label{eq:s18:summand}
  \end{equation}
  None of the terms in the denominator is divisible by $p$ since $j \geq
  1$. On the other hand, $( 2 (p - j - k)) !$ in the numerator is divisible by
  $p$ since
  \begin{equation*}
    2 (p - j - k) \geq 2 \left( p - j - \left\lfloor \frac{p - j}{3}
     \right\rfloor \right) \geq p,
  \end{equation*}
  where we used $j \leq ( p + 2) / 4$ for the final inequality.
  Therefore, for $1 \leq j \leq ( p + 2) / 4$, each of the terms in
  the sum for $s_{18} ( p - j)$ is a multiple of $p$, and we obtain the
  desired congruences.
\end{proof}

\begin{remark}
  \label{rk:s18:mod3}Employing \eqref{eq:s18:summand}, we observe that $s_{18}
  ( n) \equiv 0 \pmod{3}$ for $n \geq 1$, which
  reaffirms Corollary~\ref{cor:periodic} for this sequence.
\end{remark}

Finally, as noted in \cite{cooper-sporadic}, each of the sequences in
Table~\ref{tbl:sporadic2} times $\binom{2 n}{n}$ is an integer solution of
\eqref{eq:rec3-abcd} with $d \neq 0$. Observe that $\binom{2 n}{n}$ is
divisible by a prime $p$ for all $n$ such that $n < p \leq 2 n$. This
results in a (weaker) analog of Theorem~\ref{thm:cc} for these Ap\'ery-like
sequences, and implies, in particular, that these sequences are again
divisible by all prime numbers.

\section{Conclusion and open questions}

In Sections~\ref{sec:lucas} and \ref{sec:lucas2}, we showed that all sporadic
solutions of \eqref{eq:rec2-abc} and \eqref{eq:rec3-abcd}, given in
Tables~\ref{tbl:sporadic2} and \ref{tbl:sporadic3}, uniformly satisfy Lucas
congruences. However, for two of these sequences, especially sequence
$(\eta)$, we had to resort to a rather technical analysis. We therefore wonder
if there is a representation of these sequences from which the Lucas
congruences can be deduced more naturally, based on, for instance the
approaches of \cite{sd-laurent09} and \cite{mv-laurent13}, or
\cite{ry-diag13}. More generally, it would be desirable to have a uniform
approach to these congruences, possibly directly from the shape of the
defining recurrences and associated differential equations. In another
direction, it would be interesting to show that, as numerical evidence
suggests, {\emph{all}} of the Ap\'ery-like sequences in fact satisfy the
Dwork congruences \eqref{eq:dwork}.

The congruences \eqref{eq:A:8} show that the Ap\'ery numbers are periodic
modulo $8$, alternating between the values $1$ and $5$. As a consequence, the
other residue classes $0, 2, 3, 4, 6, 7$ modulo $8$ are never attained. On the
other hand, the observations in Section~\ref{sec:primes} show that certain
primes do not divide any Ap\'ery number. This can be rephrased as saying
that the residue class $0$ is not attained by the Ap\'ery numbers modulo
these primes. This leads us to the question of which residue classes are not
attained by Ap\'ery-like numbers modulo prime powers $p^{\alpha}$. In
particular, are there interesting cases which are not explained by
Sections~\ref{sec:periodic} and \ref{sec:primes}?

The second part of congruence \eqref{eq:Ab:half} makes it explicit that every
prime $p \equiv 3 \pmod{4}$ divides a term of the
Ap\'ery-like sequence \eqref{eq:Ab}. Is there a similarly explicit result
which demonstrates that the Ap\'ery numbers are divisible by infinitely many
distinct primes?

\section*{Acknowledgements}

This paper builds on experimental results obtained together with Arian
Daneshvar, Pujan Dave and Zhefan Wang during an Illinois Geometry Lab (IGL)
project during the Fall 2014 semester at the University of Illinois at
Urbana-Champaign (UIUC). The aim of the IGL is to introduce undergraduate
students to mathematical research. We wish to thank Arian, Pujan and Zhefan
(at the time undergraduate students in engineering at UIUC) for their great
work. In particular, their experiments predicted
Corollaries~\ref{cor:notperiodic} and \ref{cor:periodic}, and provided the
data for Table~\ref{tbl:primes}, which lead to
Conjecture~\ref{conj:prop:apery}.

We are also grateful to Eric Rowland, who visited UIUC in October 2014, for
interesting discussions on Ap\'ery-like numbers and finite state automata,
as well as for observing the congruence \eqref{eq:A:palin}.

Finally, we would like to express our gratitude to Tewodros Amdeberhan, Bruce
C.~Berndt, Robert Osburn and Wadim Zudilin for many helpful comments and
encouragement.

\end{document}